\documentclass[a4paper]{amsart}
\usepackage{graphicx}
\usepackage{etex}
\usepackage{amsthm,stmaryrd}
\usepackage{amssymb}
\usepackage{array}
\usepackage{tabu}
\usepackage{epsfig}
\usepackage[usenames,dvipsnames]{color}
\usepackage{verbatim}
\usepackage{hyperref}
\usepackage{mathrsfs}
\usepackage[all]{xy}
\usepackage{xcolor}
\usepackage{float}
\usepackage{enumerate}
\usepackage{tikz}   
\usepackage{changepage} 
\usetikzlibrary{arrows,calc,positioning,decorations.pathreplacing}

\usepackage{pst-node}
\usepackage{tikz-cd}

\newcommand{\Aut}{\operatorname{Aut}}

\newcommand{\Z}{\mathbb{Z}}

\newcommand{\C}{\mathbb{C}}
\newcommand{\N}{\mathbb{N}}

\newcommand{\U}{\mathscr{U}}

\newcommand{\qs}{q}

\definecolor{dgreen}{RGB}{0,150,0}

\newtheorem{definition}{Definition}[subsection]
\newtheorem{theorem}[definition]{Theorem}

\newtheorem*{theorem*}{Theorem}

\newtheorem{proposition}[definition]{Proposition}
\newtheorem{lemma}[definition]{Lemma}
\newtheorem{remark}[definition]{Remark}
\newtheorem{notation}[definition]{Notation}
\newtheorem{convention}[definition]{Convention}

\newtheorem{corollary}[definition]{Corollary}
\newcounter{exo} \newcounter{numexercice}
\renewcommand{\theexo}{\arabic{exo}}

\newcounter{IntroCounter}
\stepcounter{IntroCounter}

\begin{document}
\title[$U_q(sl(2))-$quantum invariants unified via intersections of Lagrangians ]{$U_q(sl(2))-$quantum invariants unified via intersections of embedded Lagrangians} 
  \author{Cristina Ana-Maria Anghel}
\address{University of Geneva, Section de mathématiques, 
Rue du Conseil-Général 7-9, Geneva, CH 1205 Switzerland} 

\email{Cristina.Palmer-Anghel@unige.ch} 
  
\thanks{ }
  \date{\today}
\begin{abstract}
In this paper we prove a unified model for $U_q(sl(2))$ quantum invariants through intersections of embedded Lagrangians in configuration spaces.
More specifically, we construct a {\em state sum of Lagrangian intersections in the configuration space in the punctured disc}, which is a polynomial in three variables. It {\em recovers the coloured Jones polynomial and the coloured Alexander polynomial} through specialisations of coefficients. This formula works for oriented links coloured with the same representation of the quantum group and can be evaluated at roots of unity. As a corollary, the Jones and Alexander polynomials come both as {\em specialisations of an intersection pairing between embedded Lagrangians} in configuration spaces, which is suitable for computations. In particular, we obtain the {\em first intersection model for the Jones polynomial} from intersections between submanifolds which are given by {\em arcs and circles} in the punctured disc.
\end{abstract}

\maketitle
\setcounter{tocdepth}{1}
\vspace{-10mm} 
 \tableofcontents
 {
\vspace{-10mm} 
\section{Introduction} 
The families of quantum invariants coming from the quantum group $U_q(sl(2))$ are powerful invariants which are conjectured to encode rich geometrical information about knot complements. More specifically the  representation theory of $U_q(sl(2))$ with generic $q$ leads to the family of invariants $\{J_N(L,q) \in \mathbb Z[q^{\pm 1}]\}_{N \in \mathbb N}$ called coloured Jones polynomials. This sequence recovers the original Jones polynomial as its first term. Dually, the quantum group at roots of unity $U_{\xi_N}(sl(2))$, leads to a sequence of non-semisimple invariants, called coloured Alexander polynomials (or ADO invariants \cite{ADO}), recovering the original Alexander polynomial.  So far, the Alexander polynomial is well understood in terms of the geometry of the knot complement but it is still a deep problem to understand the Jones polynomial by means of geometry. A first step in this direction was done by Lawrence and Bigelow in \cite{Law} and \cite{Big}. In \cite{Law1} Lawrence introduced a sequence of representations of the braid group, on the homology of coverings of configuration spaces. Based on that she and later Bigelow showed that the Jones polynomial is an intersection pairing between homology classes, called noodles and forks, in a covering of the configuration space in the punctured disc. We will call such a description a {\em topological model}. They showed also a topological model for the $A_N$ polynomials (specialisations of the HOMFLY-PT polynomial). All this work is based on the description of invariants through skein theory.
Also, in \cite{Big4}, Bigelow, Florens and Cattabriga gave a topological model for the original Alexander polynomial. A further step towards the topological understanding of quantum invariants was done by Ito, who described the loop expansions of the coloured Jones polynomials (\cite{Ito3}) as well as the coloured Alexander polynomials (\cite{Ito2}) as {\em sums of traces of homological representations}. This used identifications between quantum and homological representations of the braid group introduced by Kohno in \cite{Koh}. Further on, in \cite{Cr1} and \cite{Cr2} the author constructed {\em topological models} for the coloured Jones and coloured Alexander invariants. However, the cycles from these models had rather complicated descriptions. 

Recently, in \cite{Cr3} the author showed a {\em unified topological model} for the coloured Jones and Alexander invariants, as intersections between explicit immersed Lagrangians in the configuration spaces. This used a more explicit version of Kohno's identification provided by Martel in \cite{Martel}. These homology classes were given by lifts of linear combinations of immersed Lagrangians. However, from the computational point of view, the choices of lifts--given by paths to the base points-- were non-trivial to deal with (see figure \ref{fig:classes}).
\subsection{Results} In this paper we provide a {\em unified model} for the $U_q(sl(2))$ quantum invariants, from a {\em state sum of Lagrangian intersections}. 
\begin{itemize}
\item{ \bf [Embedded model]} The main result, Theorem \ref{THEOREM}, shows a unified state sum model over $3$ variables, $\Lambda_N(\beta_n)$, constructed from graded intersections between embedded Lagrangians in the configuration space in the punctured disk, recovering the $N^{th}$ coloured Jones and the $N^{th}$ coloured Alexander polynomials of the closure. This model is suitable for computations.
\item{\bf [Jones and Alexander polynomials]} In particular, for $N=2$ we obtain a geometric model -- constructed from graded intersections between submanifolds given by arcs and circles in the punctured disc -- which specialises to both Jones and Alexander polynomials. It is the {\em first intersection model for the Jones polynomial} obtained from {\em arcs and circles}. 

The current paper is used for a sequel \cite{Cr} where show that the above intersection model can be read from a Heegaard diagram on a surface.  
\item{\bf [Immersed model]} Along the way, in Theorem \ref{THEOREMIM}, we describe an easier immersed formula than the one from \cite{Cr3}, denoted by $\Omega_N(\beta_n)$, which is constructed with different paths to the base points, suitable for computations, recovering as well the two quantum invariants (see figure \ref{introtikz}).  
 \end{itemize}
 These two models set up a {\em framework for investigating categorification problems}.
\subsection{Description of the homological context} For $n,m,k,l \in \N$, we denote by $C_{n+l,m}$ the unordered configuration space of $m$ points in the $(n+l)$-punctured disc. Then, we define $\tilde{C}^{-k,l}_{n+l,m}$ to be a $\Z \oplus \Z$ covering of $C_{n+l,m}$, associated to a local system $\phi^{-k,l}$. Roughly speaking, this local system counts the winding number around the first $n-k$ punctures by the variable $x$ and the winding around the last $k$ punctures by its inverse $x^{-1}$ and it is trivial around the remaining $l$ punctures. In our construction, we will use the following tools:
\begin{enumerate}
\item sequence of  versions of Lawrence representations $H^{-k,l}_{n,m}$ which are $\Z[x^{\pm 1},d^{\pm 1}]$-modules and carry a $B_n-$action
(defined from the Borel-Moore homology of $\tilde{C}^{-k,l}_{n+l,m}$)
\item sequence of dual Lawrence representations $H^{-k,l,\partial}_{n,m}$ (definition \ref{D:4})\\ 
(defined using the homology relative to the boundary of the same covering) 
\item a graded intersection pairing:
\begin{equation*}
\langle , \rangle:H^{-k,l}_{n,m} \otimes H^{-k,l,\partial}_{n,m}\rightarrow \Z[x^{\pm1}, d^{\pm1}] \ \ \ \ \ \ \ \ \ \ \ \ \ \ \ \  (\text{definition }\ref{P:3'}).
\end{equation*}
\end{enumerate}
 
{\bf Homology classes} The construction of the Lagrangians is done by drawing certain arcs in the punctured disc, and then taking their product and considering its image under the quotient to the unordered configuration space. We refer to this set of curves in the punctured disc as ``geometric support''. Also, since we work in a covering space, the lifts will be prescribed by collections of vertical curves, from the base points to the arcs, called ``paths to the base points''. The intersection pairing $\langle , \rangle$ is computed from geometric intersections in the base configuration space which are graded by the local system, using the paths to the boundary (see \eqref{eq:1}). 

{\bf Case $N>2$} In this situation, we work with the configuration space of $(n-1)(N-1)+1$ points in the $(3n-1)$-punctured disc and the local system associated to: ($k=n$; $l=n-1$). More precisely, we have the collection of $2n$ black punctures and the collection of $n-1$ green punctures, which are the maximal points of the circles. With the above procedure, for $i_1,...,i_{n-1} \in \{0,...,N-1\}$  we define the two Lagrangians and consider their lifts in the covering, which give two homology classes (section \ref{51}). For the second homology class, we consider the open arcs which start and end in the green punctures (or alternatively the circles minus their maximal points). We then consider configuration spaces on these arcs and their product leads to a well defined homology class in the covering.
\vspace{-5mm}
\begin{figure}[H]
\centering
$${\color{red} \mathscr F_{i_1,...,i_{n-1}} \in H^{-n,n-1}_{2n,(n-1)(N-1)+1}} \ \ \text{ and }\ \  {\color{dgreen} \mathscr L_{i_1,...,i_{n-1}}\in H^{-n,n-1,\partial}_{2n,(n-1)(N-1)+1}}.$$
\includegraphics[scale=0.35]{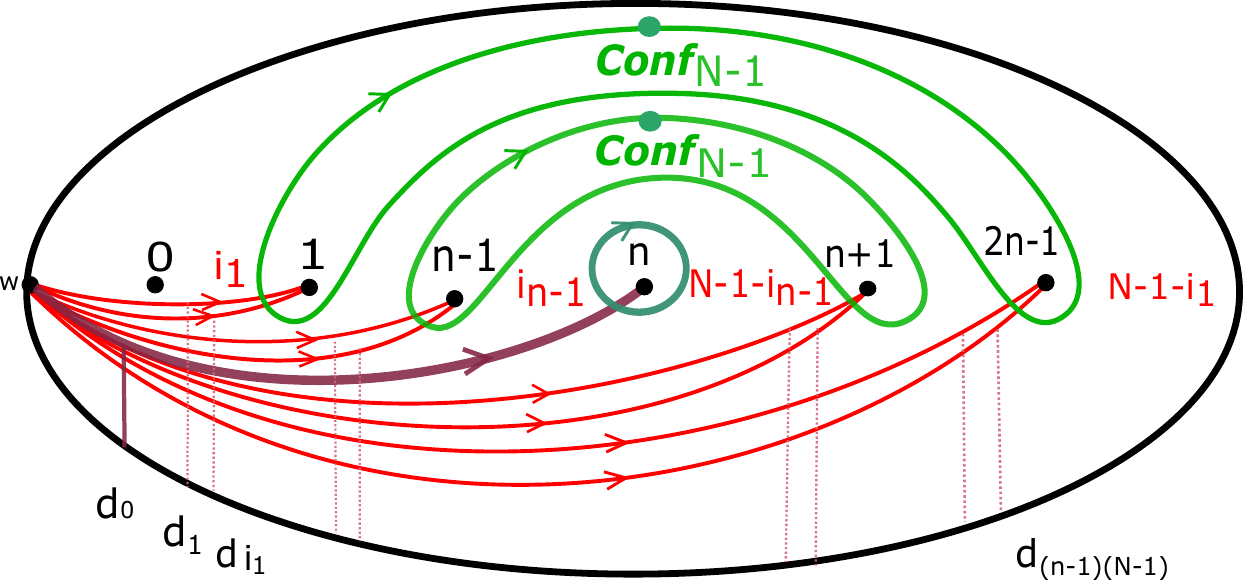}
\caption{Embedded Lagrangians: $N>2$}
\label{Picture}
\end{figure}
\vspace{-7mm}
{\bf Case $N=2$} In this case we do not remove any green punctures and consider the configuration space of $(n-1)(N-1)+1$ points in the  $2n$-punctured disc, with the local system associated to the parameters ($k=n-1$; $l=0$). For any indices $i_1,...,i_{n-1} \in \{0,1\}$  we consider the  homology classes given by the collection of red arcs and green circles from figure \ref{JAIntro} (see section \ref{52}).
\vspace{-6mm}
\begin{figure}[H]
\centering
$${\color{red} \mathscr F_{i_1,...,i_{n-1}} \in H^{-n}_{2n,(n-1)(N-1)+1}} \ \ \ \ \ \text{ and }\ \ \  \ \ \  {\color{dgreen} \mathscr L_{i_1,...,i_{n-1}}\in H^{-n,\partial}_{2n,(n-1)(N-1)+1}}.$$
\hspace{-5mm}\includegraphics[scale=0.35]{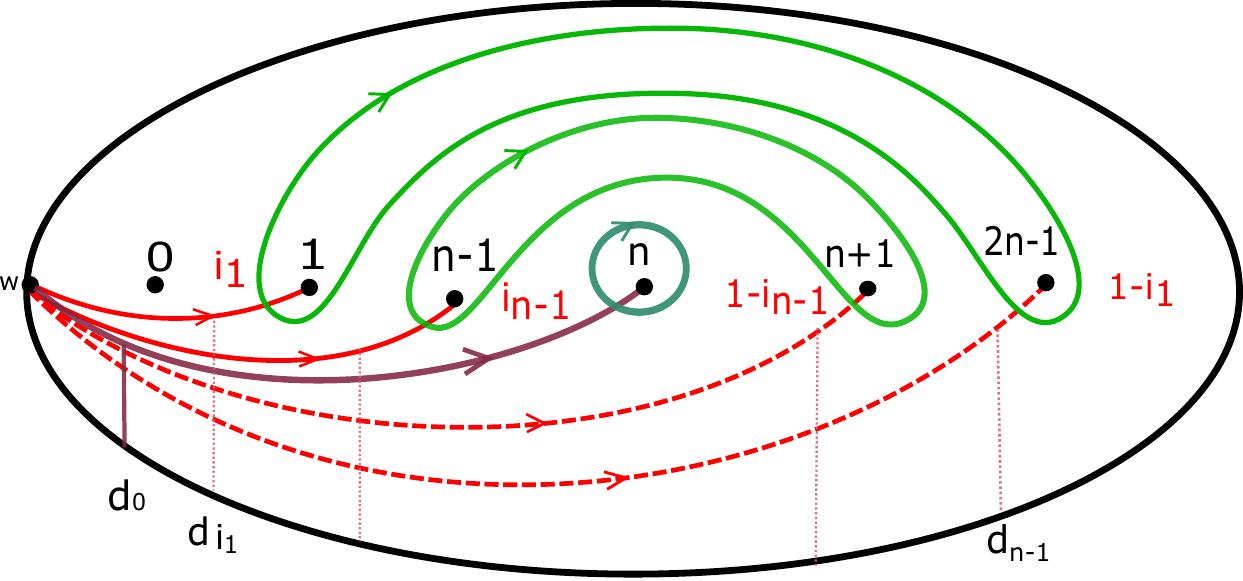}
\vspace{2mm}
\hspace{-2mm}\caption{Embedded Lagrangians: $N=2$ (Jones and Alexander polynomials)}
\label{JAIntro}
\end{figure}
\vspace{-10mm}
\subsection{Main result-Embedded state sum model}
\begin{definition}[State sum of Lagrangian intersections]

For a fixed colour $N\geqslant 2$ and $\beta_n \in B_n$, we consider the following state sum of Lagrangian intersections:
\begin{equation}\label{eq:00}
\mathscr S_N(\beta_n):=\sum_{i_1,...,i_{n-1}=0}^{N-1} \langle(\beta_{n} \cup {\mathbb I}_{n} ){ \color{red} \mathscr F_{i_1,...,i_{n-1}}}, {\color{dgreen} \mathscr L_{i_1,...,i_{n-1}}}\rangle \in \Z[x^{\pm 1},d^{\pm 1}].
\end{equation}
\end{definition}
\begin{theorem}[Unified model through a state sum of Lagrangian intersections]\label{THEOREM}
Let $L$ be an oriented link and $\beta_n \in B_n$ such that $L=\hat{\beta}_n$.  We consider the following polynomial in three variables:
\begin{equation}\label{eq:0}
\Lambda_N(\beta_n)(u,x,d):=u^{-w(\beta_n)} u^{-(n-1)} \mathscr S_N(\beta_n) \in \Z[u^{\pm1},x^{\pm 1},d^{\pm 1}].
\end{equation}
Then, $\Lambda_N$ recovers the $N^{th}$ coloured Jones and $N^{th}$ coloured Alexander invariants through certain specialisations (definition \ref{D:1''}):
\begin{equation}
\begin{aligned}
&J_N(L,q)=\Lambda_N(\beta_n)|_{\psi_{1,q,N-1}}\\
&\Phi_{N}(L,\lambda)=\Lambda_N(\beta_n)|_{\psi_{1-N,\xi_N,\lambda}}.
\end{aligned}
\end{equation}
Here, $w(\beta_n)$ is the writhe of the braid and ${\mathbb I}_n$ the trivial braid with $n$ strands.

\end{theorem}
\begin{remark}[Geometric meaning of the variables]
In this paper, we notice what the geometric role of each variable from $\Lambda_N(u,x,d)$ is:
\begin{itemize}
\item[$\bullet$]$x$ encodes the linking number with the link, as a winding number around punctures
\item[$\bullet$]$d$ on the other hand has an extra datum, encoding the twisting of the Lagrangians in the configuration space
\item[$\bullet$]$u$ captures the major difference between the semisimplicity/ non-semisimplicity of the invariant, which comes as a power of the variable $x$, so a power of the linking number with the link. 
\end{itemize}
\end{remark}
\vspace{-3 mm}
\subsection{Computable model via immersed Lagrangians} Now we present an immersed model. This differs from the one in \cite{Cr3} by the fact that it captures geometrically the coefficients of the first family of classes and also uses paths to the base points that are much simpler. This is important from the computational viewpoint. In this case we do not  consider two cases, the definition is the same for all $N\geqslant2$.
\begin{definition}[Classes suitable for computations] We consider the classes given by the geometric supports and the paths to the base points from figure \ref{PictureIM}:
\end{definition}
\vspace{-7mm}
\begin{figure}[H]
\centering
$${\color{red} \mathscr F^{I}_{i_1,...,i_{n-1}} \in H^{0}_{2n,(n-1)(N-1)+1}} \ \ \ \ \ \ \  \text{ and }\ \ \ \ \ \ \ \  {\color{dgreen} \mathscr L^{I}_{i_1,...,i_{n-1}}\in H^{0,\partial}_{2n,(n-1)(N-1)+1}}.$$
{\includegraphics[scale=0.35]{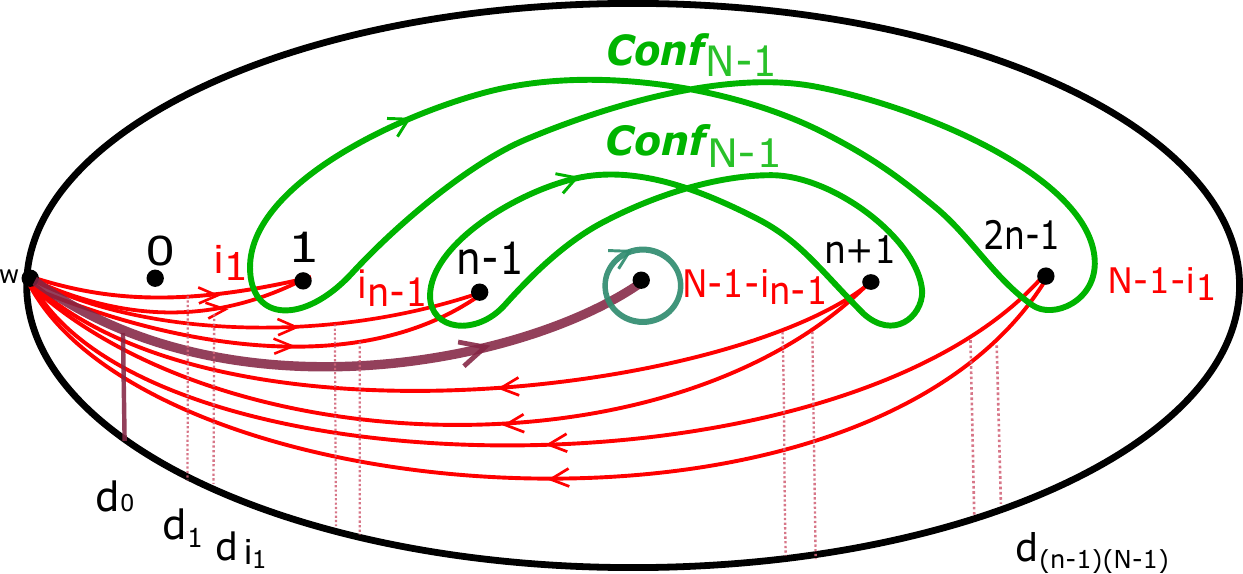}}
\vspace{2mm}
\caption{Immersed Lagrangians-computable state sum model}
\label{PictureIM}
\end{figure}
\vspace{-8mm}
\begin{theorem}[Unified model via states of immersed Lagrangian intersections]\label{THEOREMIM} We consider the following polynomial in three variables:
\begin{equation}\label{eq:0'}
\begin{aligned}
\Omega_N(\beta_n)&:=u^{-w(\beta_n)} u^{-(n-1)}\cdot\\
&\sum_{i_1,...,i_{n-1}=0}^{N-1} \langle(\beta_{n} \cup {\mathbb I}_{n} ){ \color{black} \mathscr F^{I}_{i_1,...,i_{n-1}}}, {\color{black} \mathscr L^{I}_{i_1,...,i_{n-1}}}\rangle \in \Z[u^{\pm},x^{\pm 1},d^{\pm 1}].
\end{aligned}
\end{equation}
Then, $\Omega_N$ recovers the $N^{th}$ coloured Jones and coloured Alexander polynomials:
\begin{equation}
\begin{aligned}
&J_N(L,q)=\Omega_N(\beta_n)|_{\psi_{1,q,N-1}}\\
&\Phi_{N}(L,\lambda)=\Omega_N(\beta_n)|_{\psi_{1-N,\xi_N,\lambda}}.
\end{aligned}
\end{equation}
\end{theorem}
We conclude that for any $N \in \N$ we have two intersection models: the first one  which is constructed via state sums of embedded Lagrangian intersections $\Lambda_N(\beta_n)$ and a second one, defined using states of immersed Lagrangian intersections $\Omega_N(\beta_n)$, each of which recoves the $N^{th}$ coloured Jones and the $N^{th}$ coloured Alexander polynomials via different specialisations of the coefficients.
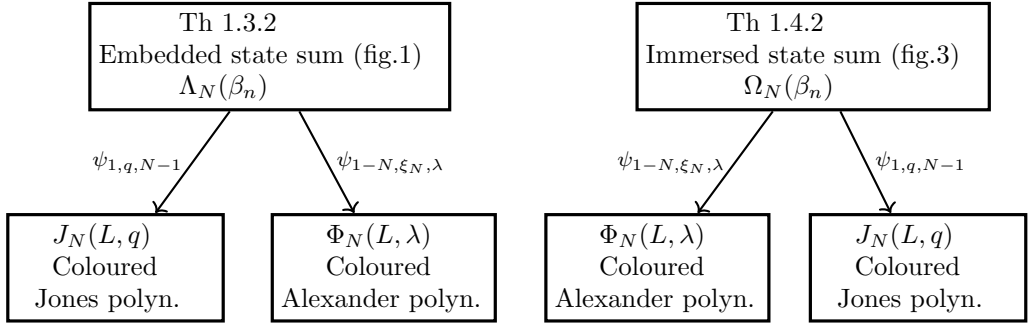
\begin{figure}[H]
\begin{center}
\begin{tikzpicture}
[x=1.2mm,y=1.4mm]

% Nodes of the diagram
\node(I)[draw,rectangle,anchor=west,very thick,text width=4.5cm,minimum height=1cm] at (-35,10) 
{\phantom \ \ \ \ \ \ \ \ \ Th \ref{THEOREM}\\
Embedded state sum (fig.\ref{Picture}) \\
\phantom  \ \ \ \ \ \ \ \ \ $\Lambda_N(\beta_n)$};
\node(E)[draw,rectangle,anchor=west,very thick,text width=4.4cm,minimum height=1cm] at (25,10) 
{\phantom \ \ \ \ \ \ \ \ \ Th \ref{THEOREMIM}\\
Immersed state sum (fig.\ref{PictureIM})\\
\phantom  \ \ $ \ \ \ \ \ \ \ \ \ \Omega_N(\beta_n)$};
\node(J)[draw,rectangle,anchor=west,very thick,text width=2.6cm,minimum height=0.1cm] at (-44,-10) 
{$ \ \ \ \ J_N(L,q)$\\
 \phantom \ \ \ \ Coloured \\
 \phantom \ \ Jones polyn.};
\node(A)[draw,rectangle,anchor=west,very thick,text width=2.7cm,minimum height=0.1cm] at (-15,-10) 
{$ \ \ \ \ \ \Phi_N(L,\lambda)$\\
\phantom \ \ \ \ \ Coloured\\
Alexander polyn.};

\node(J')[draw,rectangle,anchor=west,very thick,text width=2.6cm,minimum height=0.1cm] at (44,-10) 
{$ \ \ \ \ J_N(L,q)$\\
 \phantom \ \ \ \ Coloured \\
 \phantom \ \ Jones polyn.};
\node(A')[draw,rectangle,anchor=west,very thick,text width=2.7cm,minimum height=0.1cm] at (15,-10) 
{$ \ \ \ \ \ \Phi_N(L,\lambda)$\\
\phantom \ \ \ \ \ Coloured\\
Alexander polyn.};
\draw[->,thick] (I) to node[above,font=\footnotesize,left]{$\psi_{1,q,N-1}$} (J);
\draw[->,thick] (I) to node[above,font=\footnotesize,right]{$\psi_{1-N,\xi_N,\lambda}$} (A);
\draw[->,thick] (E) to node[above,font=\footnotesize,right]{$\psi_{1,q,N-1}$} (J');
\draw[->,thick] (E) to node[above,font=\footnotesize,left]{$\psi_{1-N,\xi_N,\lambda}$} (A');
\end{tikzpicture}
\caption{State sum models}
\label{introtikz}
\end{center} 
\end{figure}
\vspace{-10mm}
\subsection{Intersection models for Jones and Alexander polynomials}
For the particular case where $N=2$ we get both Jones and Alexander polynomials from the same picture: they come from an intersection pairing between circles and forks in the configuration space, specialised in two ways. Also, these two invariants come from a pairing between noodles and forks in the configuration space.
\begin{corollary}[Jones and Alexander polynomials from the same picture]\label{C:3}
\begin{equation}\label{eqC:3}
\begin{aligned}
&J(L,q)= \Lambda_2(\beta_n)|_{u=q;x=q^{2}, d=q^{-2}} \hspace{10mm} J(L,q)= \Omega_2(\beta_n)|_{u=q;x=q^{2}, d=q^{-2}}\\
&\Delta(L,t)=\Lambda_2(\beta_n) |_{u=t^{-\frac{1}{2}};x=t;d=-1} \hspace{9mm} \Delta(L,t)=\Omega_2(\beta_n) |_{u=t^{-\frac{1}{2}};x=t;d=-1}.
\end{aligned}
\end{equation}
\end{corollary}

\begin{remark}[Computations]
The immersed model for the Jones polynomial has the same underlying geometric support as Bigelow's model, but the homology classes are different. For example, we show that for the trefoil knot our model leads to an easy computation (figure \ref{Timmersed}), whereas Bigelow mentions that the model from \cite{Big} needs a careful check for the trefoil knot. 

In section \ref{E:Example} we show a computation of the Jones and Alexander polynomials of the trefoil knot using the embedded picture \ref{Tembedded} as well as the immersed picture \ref{Timmersed}.
\end{remark}
 \subsection{Framework for categorifications}
%The main motivation of this work concern geometrical categorifications for these quantum invariants. The embedded Lagrangians from above and their explicit form provide a proper context for this studying this question.

An active research area concerns relations between categorifications given by Khovanov homology and Heegaard Floer homology. In \cite{Ras} Rasmussen predicted a spectral sequence from Khovanov homology to knot Floer homology. This was proved recently by Dowlin (\cite{D}). However, there are still open questions about the geometric nature of this spectral sequence and relations between categorifications of these invariants. Also, Kotelskiy-Watson-Zibrowius (\cite{KWZ}) showed that Khovanov homology can be obtained from a construction using immersed curves. It would be interesting to compare the intersection model from Theorem \ref{THEOREMIM} with the Kotelskiy-Watson-Zibrowius’ immersed curve formula. Also, it would be interesting to investigate for Jones polynomial relations between the Lagrangians defined by Seidel-Smith and Manolescu (\cite{SM}, \cite{M1}, which is known to give a geometrical categorification) and our Lagrangians.

The two explicit models from this paper set up a {\em framework for investigating categorification problems for the whole families of coloured Alexander polynomials and coloured Jones polynomials}. Moreover, they provide a tool for understanding possible relations between them.

We believe that the embedded model $\Lambda_2(\beta_n)$ (Th \ref{THEOREM}) leads to knot Floer homology of the closure. On the other hand, it seems that the immersed model is better suited for the Jones polynomial.
We expect that the immersed intersection model $\Omega_2(\beta_n)$ (Th \ref{THEOREMIM}) and the embedded one $\Lambda_2(\beta_n)$ (Th \ref{THEOREM}) will lead to a spectral sequence between categorifications of Jones polynomial and knot Floer homology.

\subsection{Further development}
These topological models lead to further developments in \cite{Cr0}, where we present a topological model for the level $\mathcal N$ Witten-Reshetikhin-Turaev invariants for $3$-manifolds, as state sums of Lagrangian intersections in a fixed configuration space in the punctured disc. Further on, in \cite{Cr} we use the current paper in order to prove that the colour $N$ coloured Jones and Alexander invariants are both specialisations of a graded intersection between two fixed Lagrangians in a symmetric power of the surface.

Pursuing this line, we are interested in studying geometrically the behaviour of $\Lambda_N(\beta_{n})$ and $\Omega_N(\beta_{n})$. Gukov and Manolescu constructed in \cite{GM} a two variable power series from knot complements, and conjectured that it is a knot invariant recovering the loop expansion of $J_N$. It would be interesting to study asymptotic limits of $\Lambda_N(\beta_{n})$ and to investigate possible relations between this model and the Gukov-Manolescu power series. In \cite{W}, Willets defined algebraically a knot invariant which recovers the $N^{th}$ coloured Jones polynomial and the $N^{th}$ coloured Alexander polynomial divided by the Alexander polynomial. His formula works for knots (not for links) but it is a knot invariant. 
Our models recover exactly the $N^{th}$ coloured Jones and the $N^{th}$ coloured Alexander polynomials and work for any link. It would be interesting to compare the two results.

\subsection*{Structure of the paper} In Section \ref{3} we discuss homological tools, which are given by versions of the Lawrence representation. In Section \ref{4} we present the homological model with immersed classes from \cite{Cr3}. Section \ref{5} is devoted to the construction of the embedded Lagrangians that will be the good building blocks for our model. In Section \ref{6}, we investigate the two state sums of Lagrangian intersections and prove Theorem \ref{THEOREM}. Section \ref{7} concerns the state model via immersed Lagrangians presented in Theorem \ref{THEOREMIM}. In the last part, Section \ref{8}, we give a useful form of the intersection model and remark that the special case, which corresponds to $N=2$, provides unified models for the Jones and Alexander polynomials. We compute these for the trefoil knot.
\subsection*{Acknowledgements} 
 This paper was prepared at the University of Oxford, and I acknowledge the support of the European Research Council (ERC) under the European Union's Horizon 2020 research and innovation programme (grant agreement No 674978). For the revised version of this paper I would like to acknowledge the support of the Swiss NSF grant $200020-200400$ and SwissMAP NCCR 51NF40-141869 “The Mathematics of Physics”.
 
\section{Notations}
Throughout the paper, we will use certain specialisations of coefficients for  homology modules. We will use the following notations.
\begin{notation}\label{N:spec} 
Suppose that $R$ is a ring and $M$ an $R$-module with a fixed basis $\mathscr B$.
We consider $S$ to be another ring and suppose that we fix a specialisation of the coefficients, given by a morphism between rings:
$$\psi: R \rightarrow S.$$
Then, the specialisation of the module $M$ by the morphism $\psi$ is defined to be the following $S$-module: $$M|_{\psi}:=M \otimes_{R} S.$$ It will have a basis given by:
$$\mathscr B_{M|_{\psi}}:=\mathscr B \otimes_{R}1 \in M|_{\psi}. $$
\end{notation}
\begin{definition}[Specialisations]\label{D:1''}   For $c \in \Z$, we will use the following morphism:
\begin{equation}
\begin{aligned}
&\psi_{c,q,\lambda}: \Z[u^{\pm 1},x^{\pm1},d^{\pm1}]\rightarrow \Z[q^{\pm 1},q^{\pm \lambda}]\\
& \psi_{c,q,\lambda}(u)= q^{c \lambda}; \ \ \psi_{c,q,\lambda}(x)= \qs^{2 \lambda}; \ \ \psi_{c,q,\lambda}(d)=\qs^{-2}.
\end{aligned}
\end{equation}
\end{definition}
\begin{definition}[Specialisation of coefficients]\label{D:1}   For $\lambda \in \C$, let us consider the following change of coefficients:
\begin{equation*}
\begin{aligned}
&\psi_{q,\lambda}: \Z[x^{\pm1},d^{\pm1}]\rightarrow \Z[q^{\pm 1},q^{\pm \lambda}]\\
&\psi_{q,\lambda}(x)= \qs^{2 \lambda}; \ \ \psi_{\lambda}(d)=\qs^{-2}.
\end{aligned}
\end{equation*}
\end{definition}
Throughout the following sections, we will use the roots of unity $\xi_N=e^{\frac{2 \pi i}{2N}}$.
\section{Lawrence representation and intersection pairings}\label{3}
In this part we present the topological tools that we need for the intersection models. We will use the Borel-Moore homology of certain coverings of the configuration spaces in the punctured disc. We start with two natural numbers $n$ and $m$ and denote by $\mathscr D_n$ the two dimensional disc with boundary, with $n$ punctures: $$\mathscr D_n=\mathbb D^2 \setminus \{1,...,n\}.$$
Moreover, let $S^{-}\subseteq \partial \mathbb D^2$ be the semicircle in its boundary given by points with negative $x$-coordinate. We denote the unordered configuration space in this punctured disc as:
$$C_{n,m}:=\{ (x_1,...,x_m)\in (\mathscr D_n)^{\times m} \mid x_i \neq x_j, \forall \ 1 \leq i < j\leq m\} /Sym_m. $$
(here $Sym_m$ is the $m^{th}$ symmetric group)

Now, let us fix $d_1,..d_m \in \partial\mathscr D_n$ and ${\bf d}=(d_1,...,d_m)$ the associated base point in the configuration space.
\subsection{Local system}

\begin{remark}[Homology]
For $m \geq 2$, we use the abelianisation map: 
\begin{equation*}
\begin{aligned}
\rho: \pi_1(C_{n,m})\rightarrow H_1\left( C_{n,m}\right) \simeq & \ \ \Z^{n} \ \oplus \ \Z \ \ \\
& \langle \rho(\sigma_i) \rangle \  \langle \rho(\delta) \rangle,  \ \ \ \ \ {i\in \{1,...,n\}}.
\end{aligned}
\end{equation*}
Here, $\sigma_i\in \pi_1(C_{n,m})$ is the loop in the configuration space based in $\bf d$ whose first component is going on a loop in the punctured disk around the $i^{th}$ puncture and all the other components are fixed. The generator $\delta \in \pi_1(C_{n,m})$ is the loop whose first two components swap the points $d_1$ and $d_2$ anticlockwise and the other components are $(m-2)$ constant points ($d_3$,...,$d_{n}$) as in figure \ref{fig2}.  

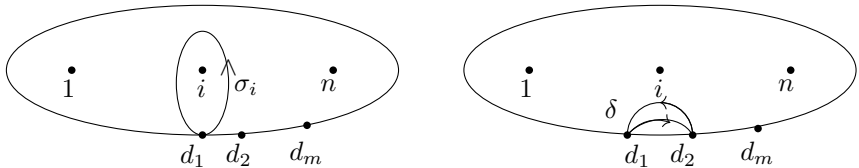
\begin{figure}[H]
\begin{tikzpicture}
[scale=4.3/5]
\foreach \x/\y in {0/2,2/2,4/2,2/1,2.6/1,3.6/1.15} {\node at (\x,\y) [circle,fill,inner sep=1pt] {};}
\node at (0.2,2) [anchor=north east] {$1$};
\node at (2.2,2) [anchor=north east] {$i$};
\node at (4.2,2) [anchor=north east] {$n$};
\node at (3,2) [anchor=north east] {$\sigma_i$};
\node at (2.2,1) [anchor=north east] {$d_1$};
\node at (2.9,1.02) [anchor=north east] {$d_2$};
\node at (4,1.05) [anchor=north east] {$d_m$};
\node at (2.68,2.3) [anchor=north east] {$\wedge$};
\draw (2,1.8) ellipse (0.4cm and 0.8cm);
\draw (2,2) ellipse (3cm and 1cm);
\foreach \x/\y in {7/2,9/2,11/2,8.5/1,9.5/1,10.5/1.10} {\node at (\x,\y) [circle,fill,inner sep=1pt] {};}
\node at (7.2,2) [anchor=north east] {$1$};
\node at (9.2,2) [anchor=north east] {$i$};
\node at (11.2,2) [anchor=north east] {$n$};
\node at (9,1) [anchor=north east] {$d_1$};
\node at (9.7,1.02) [anchor=north east] {$d_2$};
\node at (10.9,1.05) [anchor=north east] {$d_m$};
\node at (8.5,1.7) [anchor=north east] {$\delta$};
\draw (9,2) ellipse (3cm and 1cm);
%\draw[->, color=blue, very thick]             (12.2,1)     to[in=30,out=30] node[yshift=-3mm,font=\small]{}  (12.5,1.02);
\draw (9.5,1)  arc[radius = 5mm, start angle= 0, end angle= 180];
\draw [->](9.5,1)  arc[radius = 5mm, start angle= 0, end angle= 90];
\draw (8.5,1) to[out=50,in=120] (9.5,1);
\draw [->](8.5,1) to[out=50,in=160] (9.16,1.19);
\end{tikzpicture}
\caption{Local system}
\label{fig2}
\end{figure}
\end{remark}
\begin{definition}[First Local system] Let us fix a number $0 \leq k \leq n$. Then, we consider $a:\Z^n\oplus \Z\rightarrow \Z \oplus \Z$ to be the augmentation map given by:
$$a(x_1,...,x_n,y)=(x_1+...+x_{n-k}-x_{n-k+1}-...-x_{n},y).$$
Using this, we define the following local system:
\begin{equation}\label{eq:23}
\begin{aligned}
\phi^{-k}: \pi_1(C_{n,m}) \rightarrow \ & \Z \oplus \Z\\ 
 &\hspace{-1mm} \langle x \rangle \ \langle d' \rangle\\
\phi^{-k}= a \circ \rho. \ \ \ \ \ \ \ \ 
\end{aligned}
\end{equation}
\end{definition}
We will consider the homology of a certain covering of this configuration space. Our aim is to define and use pairings between homology classes given by embedded Lagrangian submanifolds, rather than the immersed ones given in \cite{Cr3}. For this, we will use geometric supports given by circles rather than figure eights. When we want to do this change there is a subtlety that appears, related to the existence of lifts of such submanifolds to the covering. In order to have such lifts, in this set-up, we add $l$ extra punctures to our punctured disc. Later on, they will have the role of removing the maximal points of the circles in order to get the well defined homology classes from section \ref{5}.

Let $l \in \N$ and consider the punctured disc $\mathscr{D}_{n+l}$ where we fix $l$ green punctures and the other $n$ black punctures as in figure \ref{fig4}. Then, we have an inclusion between the punctured discs: $$ \mathscr{D}_{n+l}\hookrightarrow \mathscr{D}_{n}$$ given by forgetting the set of $l$ green punctures. 
This induces an inclusion between the configuration spaces:
$$\iota^l: C_{n+l,m}\hookrightarrow C_{n,m},$$
which induces a map between their fundamental groups:
\begin{equation}
\tilde{\iota}^l: \pi_1(C_{n+l,m}) \rightarrow \pi_1(C_{n,m}).
\end{equation}
\begin{definition}[Second Local system] Then, we consider the local system coming from $\phi^{-k}$, via this map, and denote it as below:
\begin{equation}\label{eq:23}
\begin{aligned}
\phi^{-k,l}: \pi_1(C_{n+l,m}) \rightarrow \ & \Z \oplus \Z\\ 
 &\hspace{-1mm} \langle x \rangle \ \langle d' \rangle\\
\phi^{-k,l}= \phi^{-k} \circ \tilde{\iota}^l.
\end{aligned}
\end{equation}
\end{definition}

\begin{definition}[Covering of the configuration space]\label{D:12}
Let $\tilde{C}^{-k,l}_{n+l,m}$ be the covering of $C_{n+l,m}$ associated to the local system $\phi^{-k,l}$. 
\end{definition}
Now, we discuss the homology of this covering space. Let us fix a point $w \in S^{-} \subseteq \partial \mathscr D_n$. We will work with the part of the Borel-Moore homology of this covering which comes from the corresponding Borel-Moore homology of the base space, twisted by the local system. 

For the following definitions, we denote by $C^{-}$ the part of the boundary of $C_{n+l,m}$ represented by configurations containing a point in $S^{-}$ . Also, by $P^{-}$ we denote the part of the boundary of $\tilde{C}^{-k,l}_{n+l,m}$ represented by the fiber over $C^{-}$.
From now on, we will use the variable $d:=-d'$.

\begin{notation}\label{T2}
For the computational part, we will change slightly the infinity part of the configuration space. The details of this construction are presented in \cite{CrM} (see Remark 7.5, page 23), in the case where the underlying surface is the closed $2$-disc minus half of its boundary and minus $l$ points and $n$ open discs (with pairwise disjoint closures) from its interior. We denote the following:
\begin{equation*}
\begin{aligned}
&\bullet H^{\text{lf},\infty,-}_m(\tilde{C}^{-k}_{n+l,m},P^{-}; \Z) \text{ the homology relative to the infinity part given by }\\
& \text{ the open boundary of the covering of the configuration space consisting }\\
&\text{ in the configurations that project in the base space to a multipoint which }\\
& \text { touches a {\bf black} puncture from the punctured disc and relative to the } \\
& \text { boundary part defined by } P^{-}.\\
&\bullet H^{lf, \Delta}_{m}(\tilde{C}^{-k,l}_{n+l,m}, \partial; \Z) \text{ the homolgy relative to the boundary of } \tilde{C}^{-k,l}_{n,m} \\
& \text{ which is not in } P^{-} \text{and Borel-Moore}\\
& \text{ with respect to collisions of points in the configuration space}\\
&\text {and the open boundary where a point touches a {\bf green} puncture}\\
&\text { from the punctured disc}.
\end{aligned}
\end{equation*}
\end{notation}
 We would like to emphasise that the Borel-Moore homology of a covering space is generally different than the twisted Borel-Moore homology of the base space. For our purpose, we want to work with the homology of the covering space rather than the twisted homology of the base space, and for this we will use the the following results. 
%we use the homology classes presented above, seen in the modified version of the homology $H_{n,m}$. However, following \cite{CrM}, all relations between the homology classes still hold in this version of the homology. 
\begin{proposition}[\cite{CrM} Theorem E]\label{P:5}
There are natural injective maps:
\begin{equation}
\begin{aligned}
& \iota: H^{\text{lf},\infty,-}_m(C_{n+l,m}, C^{-}; \mathscr L_{\phi^{-k,l}})\rightarrow H^{\text{lf},\infty,-}_m(\tilde{C}^{-k,l}_{n+l,m}, P^{-1};\Z)\\
& \iota^{\partial}:H^{\text{lf},\Delta}_m(C_{n+l,m}, \partial; \mathscr L_{\phi^{-k,l}})\rightarrow H^{\text{lf},\Delta}_m(\tilde{C}^{-k,l}_{n+l,m},\partial;\Z).
\end{aligned}
\end{equation}
where $\mathscr L_{\phi^{-k}}$ is the rank $1$ local system associated to $\phi^{-k,l}$(\cite{CrM} Definition 2.7).
\end{proposition}
In the following sections, we use the parts of these homologies of the covering which come from the twisted homology of the base configuration space. 
\subsection{Lawrence representation} 

\begin{definition}[Homology groups]\label{D:4}
Based on Proposition \ref{P:5}, we introduce the following notations.
\begin{enumerate}
 \item[$\bullet$] Let $H^{-k,l}_{n,m}\subseteq H^{\text{lf},\infty,-}_m(\tilde{C}^{-k,l}_{n,m}, P^{-1};\Z)$ be the image of the map $\iota$.
 \item[$\bullet$] Also $H^{-k,l,\partial}_{n,m} \subseteq H^{\text{lf},\Delta}_m(\tilde{C}^{-k,l}_{n,m},\partial;\Z)$ will be the image of the map $\iota^{\partial}$.
\end{enumerate}
\end{definition}
Even if the definition of these homology groups is rather subtle, we will work with very explicit classes given by submanifolds in the configuration space, which we introduce in the next section.

Now we present the definition of the version of the homological Lawrence representation of the braid groups which uses the above set-up. 
\begin{convention}[Parameter which is zero] If one of the parameter $l$ or $k$ is zero, then we will remove it from the indices of the homology groups:
$$H^{-k,l}_{n,m}, H^{-k,l,\partial}_{n,m}.$$
\end{convention}
\begin{notation}[$B_n$ action] We will use that the braid group $B_{n+l}$ is the mapping class group of the punctured disc $\mathscr D_{n+l}$. For our purpose, we do not need to act non-trivially on the set of $l$ green punctures, so we consider the inclusion:
$$B_{n}\hookrightarrow B_{n+l}$$ given by adding $l$ trivial strands, associated to the $l$ green punctures. All braid group actions from now on will be pre-composed with this inclusion.
\end{notation}
\begin{proposition}[Version of the Lawrence representation] \label{T1}
Following \cite{CrM}, there is a well defined braid group action (arising from the mapping class group of the punctured disc) which commutes with the action of deck transformations at the homological level:
$$B_n \curvearrowright H^{-k,l}_{n,m} \ (\text{as a module over }\Z[x^{\pm1}, d^{\pm1}]).$$ 
We denote this action by:
$$L_{n,m}: B_n\rightarrow \Aut_{\Z[x^{\pm 1},d^{\pm 1}]}(H^{-k,l}_{n,m}).$$
and call it the Lawrence representation associated to the covering space $\tilde{C}^{-k,l}_{n+l,m}$. 
\end{proposition}
\subsection{Homology classes}
More specifically, we present a recipe to produce certain homology classes in the homology of the covering space using submanifolds in the base configuration space. 

Let us fix ${\bf \tilde{d}} \in \tilde{C}^{-k,l}_{n+l,m}$ a lift of the base point ${ \bf d}=\{d_1,...,d_m\}$ in the covering.
\begin{notation}[Lifts of paths to the covering]
Let $\gamma$ be a path in $C_{n,m}$ which starts in $\bf d$. For the rest of the paper, we denote by $\tilde{\gamma}$ the unique lift of $\gamma$ in the covering $\tilde{C}^{-k,l}_{n+l,m}$ which starts in $\bf \tilde{d}$.
\end{notation}

\begin{notation}
Let us consider the following indexing set:
$$E_{n,m}=\{e=(e_1,...,e_{n})\in \N^{n} \mid e_1+...+e_{n}=m \}.$$
\end{notation}
Now we present a method for constructing homology classes. For each partition $e\in E_{n,m}$, we consider $m$ segments between the boundary point $w$ and the punctures ${1,...,n}$, which are specified by the components of the partition, as in picture \ref{fig4}. We take $\bar{U}_e$ to be the $m$-dimensional submanifold in $C_{n+l,m}$ constructed from the product of these red segments, quotiented by the action of the symmetric group. Moreover, we fix a set of $m$ paths from the boundary points to each such red segment. The set of these paths will give a path in the configuration space, which we denote by $\eta_e$. Now, we lift these to the covering. We consider $\tilde{U}_e$ to be the lift of the $m$-manifold $\bar{U}_e$ in $\tilde{C}^{-k,l}_{n+l,m}$ through $\tilde{\eta}_e(1)$.
\vspace{-3mm}
\begin{figure}[H]
\begin{center}
\begin{tikzpicture}\label{pic}
[x=0.5mm,y=0.02mm,scale=0.1/5,font=\Large]
\foreach \x/\y in {-1.2/2, 0.4/2 , 1.3/2 , 2.5/2 , 3.6/2 } {\node at (\x,\y) [circle,fill,inner sep=1.3pt] {};}
\foreach \x/\y in {2/2.6, 2/3.2 } {\color{dgreen} \node at (\x,\y) [circle,fill,inner sep=1.3pt] {};}
\node at (-1,2) [anchor=north east] {$w$};
\node at (0.5,2.6) [anchor=north east] {$1$};
\node at (4,2.6) [anchor=north east] {$n$};
\node at (3.3,3) [anchor=north east] {\color{dgreen}$n+1$};
\node at (3.2,3.5) [anchor=north east] {\color{dgreen}$n+l$};
\node at (0.1,2.6) [anchor=north east] {$\color{black!50!red}e_1$};
\node at (3.5,2.4) [anchor=north east] {$\color{black!50!red}e_{n}$};
\node at (1.3,3) [anchor=north east] {\huge{\color{black!50!red}$\bar{U}_e$}};
%\node at (1.6,3) [anchor=north east] {\huge{\color{black!50!red}$\bar{U}_e$}};
%\node at (1.6,3) [anchor=north east] {\huge{\color{black!50!red}$\bar{U}_e$}};
\node at (1.3,6.2) [anchor=north east] {\huge{\color{black!50!red}$\tilde{U}_e$}};
\node at (2,5.1) [anchor=north east] {\Large \color{black!50!red}$\tilde{\eta}_e$};
\node at (-2.5,2) [anchor=north east] {\large{$C_{n+l,m}$}};
\node at (-2.5,6) [anchor=north east] {\large{$\tilde{C}^{-k,l}_{n+l,m}$}};
\draw [very thick,black!50!red,-](-1.2,2) to (0.4,2);
\draw [very thick,black!50!red,-][in=-155,out=-30](-1.2,2) to (0.4,2);
\draw [very thick,black!50!red,-] [in=-145,out=-30](-1.2,2) to (3.6,2);
\draw [very thick,black!50!red,-] [in=-130,out=-40](-1.2,2) to (3.6,2);

\draw [very thick,black!50!red,->](-1.2,2) to (-0.5,2);
\draw [very thick,black!50!red,->][in=-160,out=-30](-1.2,2) to (-0.3,1.8);
\draw [very thick,black!50!red,->] [in=-183,out=-33](-1.2,2) to (1.4,1.25);
\draw [very thick,black!50!red,->] [in=-181,out=-37](-1.2,2) to (1.4,1);

\draw (2,2.1) ellipse (3.2cm and 1.45cm);
\draw (2,5.6) ellipse (3cm and 1.3cm);
\node (d1) at (1.3,0.7) [anchor=north east] {$d_1$};
\node (d2) at (1.8,0.7) [anchor=north east] {$d_2$};
\node (dn) at (2.8,0.7) [anchor=north east] {$d_m$};
\node (dn) at (4,1.7) [anchor=north east] {\Large \color{black!50!red}$\eta_e$};
\draw [very thick,dashed, black!50!red,->][in=-60,out=-190](1.2,0.7) to  (-0.3,2);
\draw [very thick,dashed,black!50!red,->][in=-70,out=-200](1.3,0.7) to (0,1.9);
\draw [very thick,dashed,black!50!red,->][in=-90,out=0](2.5,0.7) to (3,1.5);
\node at (0.8,0.7) [anchor=north east] {\bf d$=$};
\node at (0.8,4.4) [anchor=north east] {\bf $\bf \tilde{d}$};
\draw [very thick,dashed,black!50!red,->][in=-70,out=-200](0.8,4.4) to (3,5);
\node at (0.8,4.4) [anchor=north east] {\bf $\bf \tilde{d}$};
 \draw[very thick,black!50!red] (2.82, 5.6) circle (0.6);
\end{tikzpicture}
\end{center}
\caption{Classes in the covering space}
\label{fig4}
\end{figure}
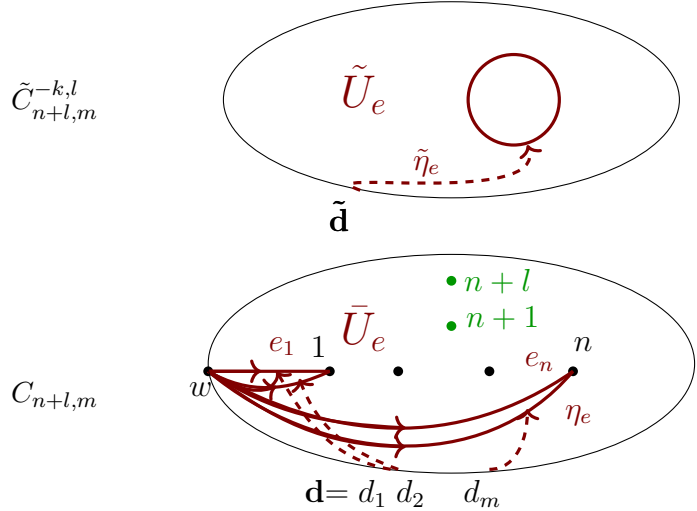
\vspace{-10mm}
\begin{definition}[Code sequence] Let us denote the homology class associated to this submanifold by:
$$\U_{e}:= [\tilde{U}_e]  \in H^{-k,l}_{n,m}.$$
\end{definition}
\subsection{Intersection pairing}
Further on, we will use a Poincar\'e-Lefschetz duality between these homologies of the covering space with respect to different parts of its boundary. 
\begin{proposition}(\cite{CrM})\label{P:3'}
For any $n,m\in \N$ and $0\leq k \leq n$, there exists a well defined intersection pairing:
$$\langle ~,~ \rangle: H^{-k,l}_{n,m} \otimes H^{-k,l,\partial}_{n,m}\rightarrow\Z[x^{\pm1}, d^{\pm1}]$$
which is sesquilinear with respect to the involution $\iota: \Z[x^{\pm 1},d^{\pm 1}]\rightarrow \Z[x^{\pm 1},d^{\pm 1}]$ given by $\iota(x)=x^{-1}$ and $\iota(d)=d^{-1}$.
\end{proposition}
In the next part we present the formula for this intersection pairing. More precisely, we will see that if we start with two classes given by lifts of $m$-dimensional submanifolds in the base configuration space, then their pairing is encoded by the geometrical intersection in $C_{n+l,m}$ together with the local system $\phi^{-k}$. 

Let us start with $C_1 \in H^{-k,l}_{n,m}$ and $C_2 \in H^{-k,l,\partial}_{n,m}$. We suppose that there exist immersed submanifolds $M_1,M_2 \subseteq C_{n,m}$ such that two of their lifts, which we denote by $\tilde{M}_1, \tilde{M}_2$, give the homology classes $C_1$ and $C_2$ respectively. Moreover, we suppose that $M_1$ and $M_2$ intersect transversely in a finite number of points. 

\begin{proposition}[Pairing from curves in the base space and the local system]\label{P:3} 

For each intersection point $x \in M_1 \cap M_2$, we will construct a loop and denote it by $l_x \subseteq C_{n,m}$. Also, let $\alpha_x$ be the sign of the geometric intersection between $M_1$ and $M_2$ in the base configuration space, at the point $x$.\\
a) {Construction of $l_x$}\\
 We assume that there are two paths $\gamma_{M_1}, \gamma_{M_2}$ in $C_{n+l,m}$ which start in $\bf d$ and end on $M_1$ and $M_2$ respectively such that
$\tilde{\gamma}_{M_1}(1) \in \tilde{M}_1$ and $ \tilde{\gamma}_{M_2}(1) \in \tilde{M}_2$.
Further on, we consider $\delta_{M_1}, \delta_{M_2}:[0,1]\rightarrow C_{n,m}$ such that:
\begin{equation}
\begin{cases}
\mathrm{Im}(\delta_{M_1})\subseteq M_1; \delta_{M_1}(0)=\gamma_{M_1}(1);  \delta_{M_1}(1)=x\\
\mathrm{Im}(\delta_{M_2})\subseteq M_2; \delta_{M_2}(0)=\gamma_{M_2}(1);  \delta_{M_2}(1)=x.
\end{cases}
\end{equation}
Then, we consider the following loop, based in $\bf d$, associated to the intersection point $x$:
$$l_x=\gamma_{M_1}^{-1}\circ\delta_{M_1}^{-1}\circ \delta_{M_2}\circ \gamma_{M_2}.$$
b) {Formula for the intersection form (\cite{Big})}\\
Then, the intersection pairing from proposition \ref{P:3'} has the following formula (which uses the loop $l_x$ and the local system):
\begin{equation}\label{eq:111}  
\langle [\tilde{M}_1],[\tilde{M}_2]\rangle=\sum_{x \in M_1 \cap M_2} \alpha_x \cdot \phi^{-k,l}(l_x) \in \Z[x^{\pm1}, d'^{\pm1}].
\end{equation}
\end{proposition}
\begin{remark}[Computation of the intersection pairing in our cases]
For our situation, we will have submanifolds $M_1,M_2 \subseteq C_{n+l,m}$ which are given by geometric supports, meaning sets of $m$ curves in the punctured disk.
In this case, we can compute the pairing by looking at product of the signs of the local intersections in the punctured disk and denote that by $\alpha_x$ (instead of the geometric intersection in the configuration space) with the price of replacing the variable $d'$ by $d=-d'$ in the formula \eqref{eq:111}:
\begin{equation}\label{eq:1}  
\langle [\tilde{M}_1],[\tilde{M}_2]\rangle=\sum_{x \in M_1 \cap M_2} \alpha_x \cdot \phi^{-k,l}(l_x) \in \Z[x^{\pm1}, d^{\pm1}].
\end{equation}
The difference comes from a sign of an induced permutation which should be counted in the formula \eqref{eq:111} and here we encode this by changing the variable $d'$ to $d$ (see the details in Section 3 from \cite{Big}).
Overall the signs $\alpha_x$ from formulas \eqref{eq:111} and \eqref{eq:1} might be different, but this is exactly compensated by changing the variables (which has the effect of modifying $\phi^{-k,l}(l_x)$ by the same sign). 
\end{remark}
We remark that the paths to the base point $\bf d$ (which we fixed in the base space) correspond to a choice of a fundamental class up to a sign, which is needed for the Poincar\'e-Lefschetz type duality presented above. However, there is still the sign ambiguity concerning the orientation of this class. This corresponds to the choice of orientation of the loop $l_x$, when it is evaluated by the local system in formula \eqref{eq:1}. We choose the above orientation, such that it fits with the choice of orientation needed in \cite{Cr3}.

\section{Topological model with immersed classes}\label{4}
In this section we present the context and the topological model for coloured Jones and Alexander polnomials constructed in \cite{Cr3}. It uses the homology with a local system with no additional twisting part and no green punctures (corresponding to $k=0$, $l=0$) and two homology classes which come from immersed submanifolds in the base configuration space. 
Let us start with an oriented link which can be seen as a closure of a braid with $n$ strands. Also, let us fix $N \in \N$ to be the colour of the invariants that we want to understand. 
\begin{notation}
In this context, we will use the following homology groups:
\begin{equation*}
H^{0,0}_{2n-1,(n-1)(N-1)} \text { \ \ \ \ and \ \ \ \ \ } H^{0,0,\partial}_{2n-1,(n-1)(N-1)}.
\end{equation*}
From now on, if we have $k=0$, meaning that our local system does not distinguish between the punctures of the punctured disc, we remove this notation from the corresponding homology groups. 
\end{notation}
In the next part, we want to define two homology classes in these groups. The first class $\mathscr E_n^{N}$ is constructed using the building blocks $\bar{U}_e$ for $e \in E_{2n-1,(n-1)(N-1)}$ in the base configuration space, lifted in a very precise manner in the covering. 
\begin{figure}[H]
\centering
%$${\color{red} \tilde{\mathscr U}_{0,i_1,...,i_{n-1},N-1-i_{n-1},...,N-1-i_{1}} \in H_{2n-1,(n-1)(N-1)}} \ \ \text{ and }\ \  {\color{dgreen}\mathscr G_{n}^N \in H^{\partial}_{2n-1,(n-1)(N-1)}}.$$
\includegraphics[scale=0.5]{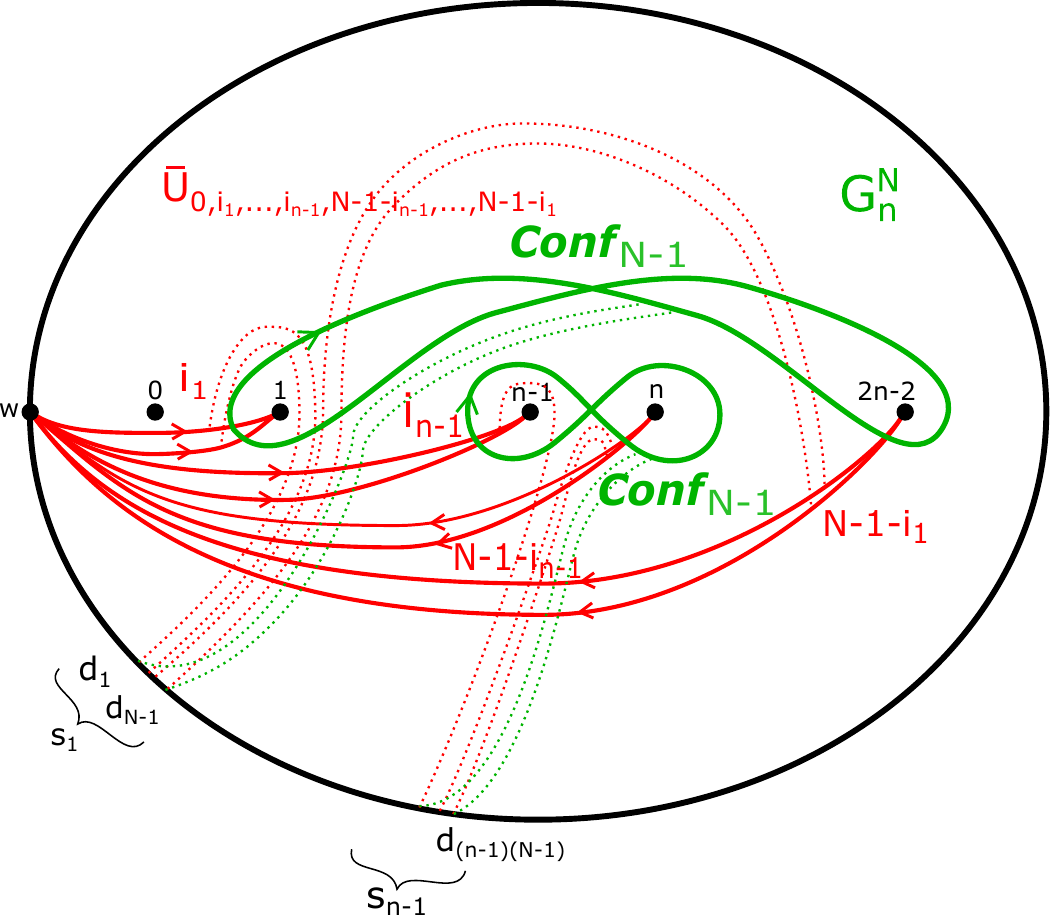}
\caption{Immersed submanifolds}
\label{fig5}
\end{figure}
\vspace{-35mm}
$\hspace{35mm} {\color{red}\eta^{U}_{i_1,...,i_{n-1}}} \hspace{15mm} {\color{dgreen}\eta^{G}} \hspace{10mm}$
\vspace{25mm}
\begin{definition}[First homology class]\label{D:2}
\noindent
\begin{itemize} 
\item[$\bullet$] Let us fix a set of indices $i_1,...,i_{n-1}\in \{0,...,N-1\}$. We start with the submanifold $\bar{U}_{{0,i_1,...,i_{n-1},N-1-i_{n-1},...,N-1-i_{1}}}$ in $C_{2n-1,(n-1)(N-1)}$. Then, we consider the path $\eta^{U}_{i_1,...,i_{n-1}}$ from the base point $\bf d$ to this submanifold, as in picture \ref{fig5}. We will use $\tilde{\eta}^{U}_{i_1,...,i_{n-1}}$ to lift this submanifold.\\ 
\item[$\bullet$]Let $\tilde{\mathscr U}_{0,i_1,...,i_{n-1},N-1-i_{n-1},...,N-1-i_{1}} \in H_{2n-1,(n-1)(N-1)}$ be the class given by the lift of $\bar{U}_{{0,i_1,...,i_{n-1},N-1-i_{n-1},...,N-1-i_{1}}}$ in $\tilde{C}_{2n-1,(n-1)(N-1)}$ through $\tilde{\eta}^{U}_{i_1,...,i_{n-1}}(1)$.
\end{itemize}
Our first class $\mathscr E_n^N \in H_{2n-1,(n-1)(N-1)}$ is defined as a linear combination of the above classes, over all choices of indices:
\begin{equation}\label{eq:1'}
\mathscr E_n^{N}=\sum_{{ i_1,...,i_{n-1}=0}}^{N-1} \hspace{-3mm} d^{- \sum_{k=1}^{n-1} i_k} \cdot \tilde{ \mathscr U}_{0,i_1,...,i_{n-1},N-1-i_{n-1},...,N-1-i_{1}}.
\end{equation}
\end{definition}
\begin{definition}[Second homology class]\label{D:3}
\noindent
\begin{itemize} 
\item[$\bullet$] Let $G_n^N$ be the immersed submanifold in $C_{2n-1,(n-1)(N-1)}$ given by the product $n-1$ ordered configuration spaces of $N-1$ points on the figure eights from picture \ref{fig5}, quotiented by the action of the symmetric group. 
\item[$\bullet$] Let $\eta^{G}$ be the path from the base point $\bf d$ to this submanifold and $\tilde{\eta}^{G}$ the associated lift.\\ 
\end{itemize}
The second homology class $\mathscr G_n^N \in H^{\partial}_{2n-1,(n-1)(N-1)}$ will be the class given by the lift of $G_n^N$ through $\tilde{\eta}^{G}(1)$.

\end{definition}
\begin{remark}
In this construction, it is important that the local system evaluates the loops around symmetric punctures $(k, 2n-1-k)$ with the same value. This means that when we go around one figure eight, it lifts to a closed loop in the corresponding covering space. This explains that the immersed manifold $G_n^N$ lifts to a closed embedded manifold in the covering $\tilde{C}_{2n-1,(n-1)(N-1)}$.
\end{remark}
\vspace{-4mm}
$$ {  \Huge \color{black} H_{2n-1,(n-1)(N-1)}}  \ \ \ \ \ \ \ \ \ \ \ \ \ \ \ \ \ \ \ \ \ \ \ \ \ \ \ \ \ \ \ \ \ \ \ \ \ \ \ \ \ \ { \Huge \color{black} H^{\partial}_{2n-1,(n-1)(N-1)}}$$
\vspace{-3mm}
$$ \ \ \ {\Huge \color{red} \mathscr E_n^{N}= \hspace{-3mm}\sum_{{ i_1,...,i_{n-1}=0}}^{N-1} \hspace{-3mm} d^{-\sum_{k=1}^{n-1} i_k} \cdot \tilde{ \mathscr U}_{0,i_1,...,i_{n-1},N-1-i_{n-1},...,N-1-i_{1}}} \ \ \ \ \ \ \ \ \ \ \ \ \ \ \ \  {\color{dgreen} \mathscr G_{n}^N} \ \ \ \ \ \ \ \ \ \ \ \ \ \ \ \ \ \ \ \ \ \ \ \ \ \ \ \ \ \ \ \ \ \ \  $$
\vspace{-5mm}

$\hspace{60 mm} \downarrow \text{lifts}$

\vspace{-3mm}
\begin{figure}[H]
\centering
\includegraphics[scale=0.5]{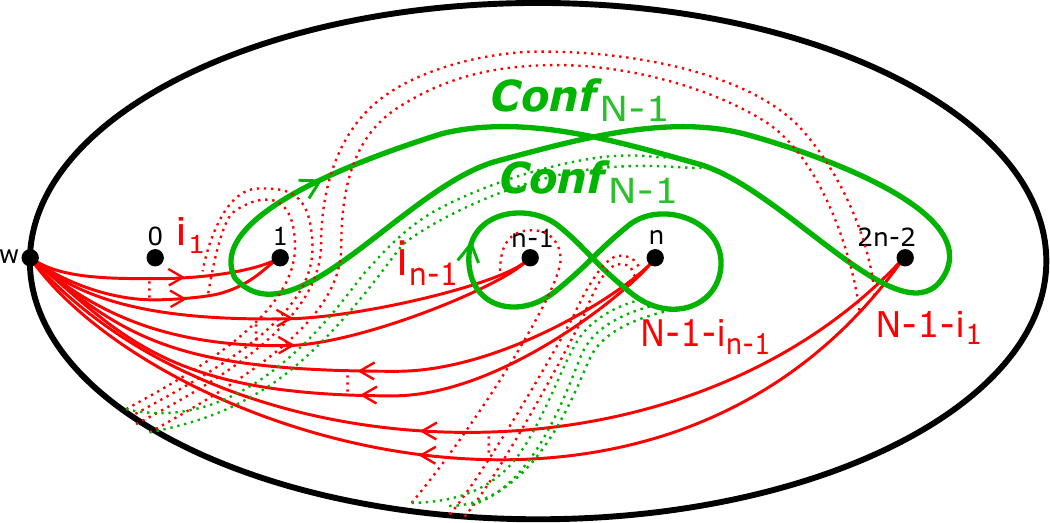}
%\vspace{-2mm}
%$${\color{red} \bar{U}_{0,i_1,...,i_{n-1},N-1-i_{n-1},...,N-1-i_{1}}} \ \ \ \ \ \ \ \ \ \ \ \ \ \ \ \ \ \ \ \ \ \ \ \ \ \ \ \ \ \ \ \ \ \ \ \ \ \ \ \  \ \ \ \ \ \ \ \ \ \ \color{dgreen} {G_n^N} \ \ \ \ \ \ \ \ \ \ \ \ \ \ \ \ \ \ \ \ \ \ \ $$
\vspace{1mm}
\caption{Homology classes-immersed model}
\label{fig:classes}
\end{figure}
\vspace{-7mm}
The main result of \cite{Cr3} shows that we can recover the two sequences of quantum invariants from the pairing between $(\beta_{n} \cup {\mathbb I}_{n-1}) \ \mathscr E_n^{N}$ and $\mathscr G_n^{N}$. 
\begin{theorem}(\cite{Cr3}) \label{T:unified}
Let $L$ be an oriented link and $\beta_n \in B_n$ such that $L=\hat{\beta}_n$. For a fixed colour $N \in \N$, we define the following polynomial in two variables:
\begin{equation}\label{eq:2}
\mathscr I_N(\beta_n):=\langle(\beta_{n} \cup {\mathbb I}_{n-1}){ \color{black} \mathscr E_n^N}, {\color{black} \mathscr G_n^N}\rangle \in \Z[x^{\pm 1},d^{\pm 1}].
\end{equation}
Then, $\mathscr I_N$ recovers the $N^{th}$ coloured Jones and Alexander invariants as below:
\begin{equation}
\begin{aligned}
& J_N(L,q)= q^{-(N-1)w(\beta_n)} \cdot \ q^{-(N-1)(n-1)} \ \ \mathscr I_N(\beta_n)|_{\psi_{q,N-1}}\\
&\Phi_{N}(L,\lambda)={\xi_N}^{-(1-N)\lambda w(\beta_n)} \cdot {\xi_N}^{-\lambda (1-N)(n-1)}  \mathscr I_N(\beta_n) |_{\psi_{\xi_N,\lambda}}.
\end{aligned}
\end{equation}
\end{theorem}
\section{Construction of the homology classes}\label{5}
The goal of this part is to construct new homology classes which will lead to the topological model with embedded Lagrangians.
\begin{itemize}
\item[$\bullet$] First, we want to replace the immersed figure eights from the above section with embedded arcs, which will be given by circles with one point that is removed. One thing that we bear in mind is that we would like to get well-defined lifts in the corresponding covering space. For this reason, we will add $(n-1)$ green punctures, and change the local system from $\phi^0$ to $\phi^{-n,n-1}$. 
\item[$\bullet$] The second aim is to construct new paths to the base point, that determine the lifts, which are simpler from the geometric point of view and such that the intersection pairing is easier to compute. 
\item[$\bullet$] The last aim is to encode geometrically the coefficients which appear in the homology class $\mathscr E_n^N$ from relation \eqref{eq:1'}. In order to do this, we add an extra puncture to the punctured disk and use configuration spaces with one extra particle. 
\end{itemize}

\subsection{Colour $N>2$}\label{51} We want to construct a model which uses geometric supports which are given by embedded submanifolds rather than figure eights.
If we replace each figure eight by a circle and consider the submanifold given by products of configurations of $N-1$ points on these circles, then this submanifold does not lift to the covering if $N>2$ (for any local system $\phi^{-k,0}$ where $k \in \N$). The reason is that even if we balance the monodromy around the symmetric punctures (which ensures that the local system evaluates to a monomial whose power of $x$ is zero), if we have more than two particles then there is a loop in the configurations on circles which has a non-trivial relative winding of the particles (see the left hand side of picture \ref{Picturemon}). So, its evaluation by the local system has a non-trivial power of $d$.

This shows that the green submanifold in the configuration space given by configurations on circles does not lift to a well-defined submanifold in the covering associated to a local system $\phi^{-k,0}$.
\begin{figure}[H]
\centering
\hspace{-5mm}\includegraphics[scale=0.37]{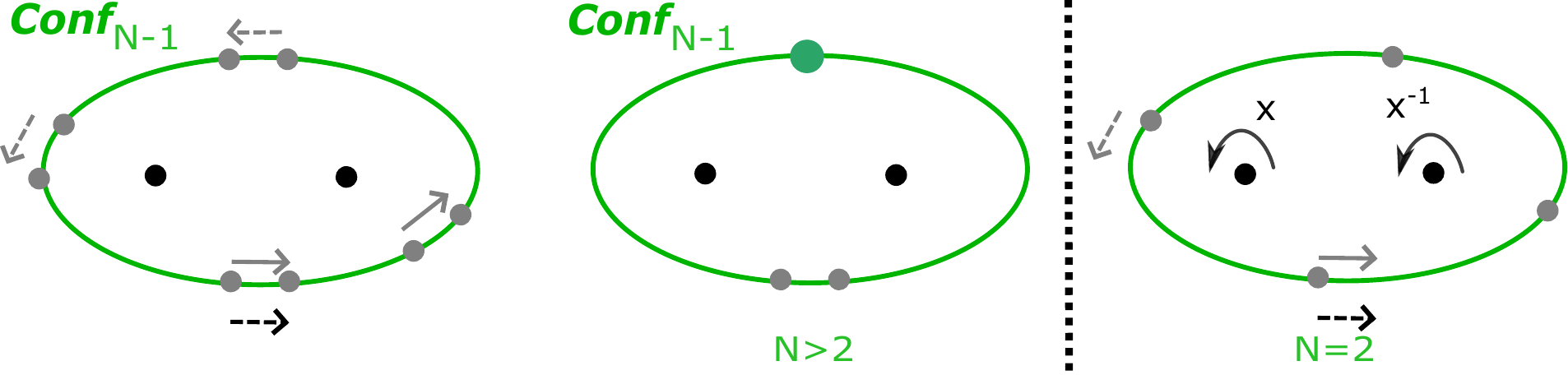}
\caption{Monodromy of configurations on circles}
\label{Picturemon}
\end{figure}

Instead, the idea is to consider configurations on circles where we remove their maximal point, in other words, we consider the open arc which starts and ends in this maximal point. 

\begin{notation}[Topological context]
In this case we work with the configuration space of $(n-1)(N-1)+1$ points in the punctured disk $\mathscr D_{3n-1}$, the local system associated to the parameters $k=n$ and $l=n-1$ and the following homology groups:
\begin{equation*}
H^{-n,n-1}_{2n,(n-1)(N-1)+1} \text { \ \ \ \ and \ \ \ \ \ } H^{-n,n-1,\partial}_{2n,(n-1)(N-1)+1}.
\end{equation*}
\end{notation}
\clearpage
\begin{definition}[First homology class]

\

Let us fix a set of indices $i_1,...,i_{n-1}\in \{0,...,N-1\}$. 

\begin{itemize}
\item[$\bullet$] We consider $\bar{V}_{i_1,...,i_{n-1}}$ to be the image of the product of all the red segments and the purple segment from figure \ref{Picture11} in the configuration space $C_{3n-1,(n-1)(N-1)+1}$. 
\item[$\bullet$]Also, we denote by $\eta^V_{i_1,...,i_{n-1}}$ the path in the configuration space from the base point $\bf d$ towards $\bar{V}_{i_1,...,i_{n-1}}$.
\end{itemize}
Further on, we consider the class given by the lift of $ \bar{V}_{i_1,...,i_{n-1}}$ through $\tilde{\eta}^V_{i_1,...,i_{n-1}}(1)$ and denote it as:
$$\mathscr F_{i_1,...,i_{n-1}} \in H^{-n,n-1}_{2n,(n-1)(N-1)+1}.$$
\end{definition}
\begin{definition}[Second homology class]

\
 
\begin{itemize}
\item[$\bullet$] Let $\bar{L}$ be the image of the product of configuration spaces of $N-1$ points on the green circles (from figure \ref{Picture11}) times the one dimensional submanifold given by the blue circle, quotiented into the unordered configuration space $C_{3n-1,(n-1)(N-1)+1}$. 
\item[$\bullet$]Also, for a set of indices $i_1,...,i_{n-1}\in \{0,...,N-1\}$
we denote by $\eta^L_{i_1,...,i_{n-1}}$ the path in the configuration space from the base point $\bf d$ towards $\bar{L}$, presented in figure \ref{Picture11}.
\end{itemize}
We consider the class given by the lift of $ \bar{L}$ through $\tilde{\eta}^L_{i_1,...,i_{n-1}}(1)$ and denote it by:
$$\mathscr L_{i_1,...,i_{n-1}} \in H^{-n,n-1,\partial}_{2n,(n-1)(N-1)+1}.$$
\end{definition}
\begin{proposition}[Configurations on circles for $N>2$]

\
 
The class $\mathscr L_{i_1,...,i_{n-1}}$ is well defined in the dual homology $H^{-n,n-1,\partial}_{2n,(n-1)(N-1)+1}$.
\end{proposition}
\begin{figure}[H]
\centering
$${\color{red} \mathscr F_{i_1,...,i_{n-1}} \in H^{-n,n-1}_{2n,(n-1)(N-1)+1}} \ \ \ \ \ \ \ \ \ \text{ and }\ \ \ \ \ \ \  \ \ \  {\color{dgreen} \mathscr L_{i_1,...,i_{n-1}}\in H^{-n,n-1,\partial}_{2n,(n-1)(N-1)+1}}.$$
\hspace{-5mm}\includegraphics[scale=0.37]{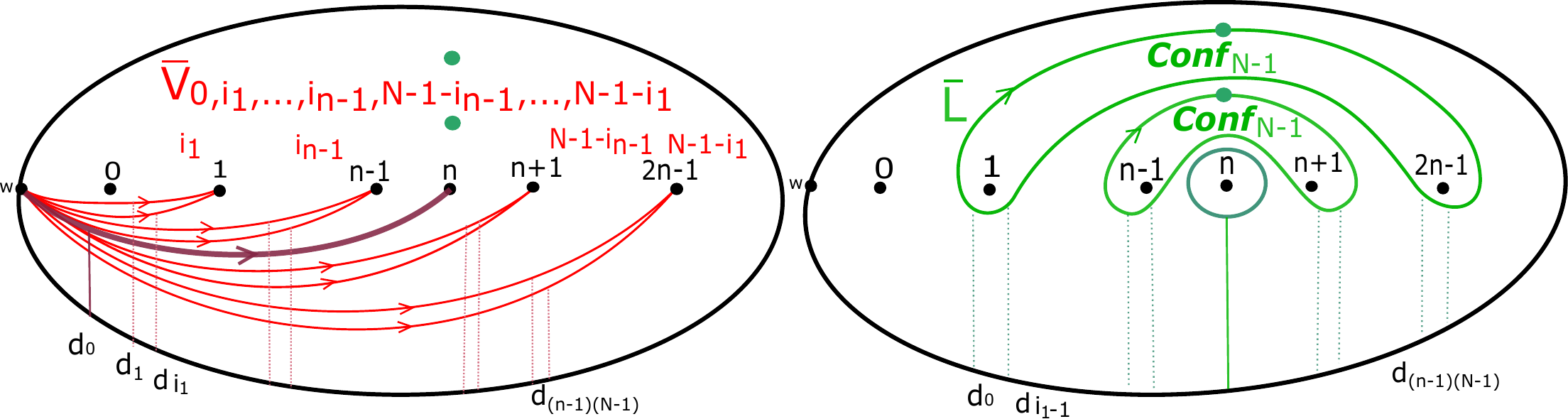}
\vspace{-13mm}
$$\hspace{15mm} {\color{purple!70!pink}\eta^{V}_{i_1,...,i_{n-1}}} \hspace{60mm} {\color{dgreen}\eta^L_{i_1,...,i_{n-1}}} \hspace{14mm}$$
%\vspace{-10mm}
\vspace{2mm}
\caption{Embedded classes $N>2$}
\label{Picture11}
\end{figure}
\begin{proof}
This comes from the fact that the geometric support of $\bar{L}$ gives a submanifold in the configuration space which has a well-defined lift in the covering $\tilde{C}^{-n,n-1}_{3n-1,(n-1)(N-1)+1}$ (since it is contractible, thus in particular simply-connected). This submanifold, in turn, leads to a well-defined class in the dual homology of the covering of the configuration space, which is Borel-Moore also relative to configurations which approach one of the $n-1$ maximal green punctures.
\end{proof}
\subsection{Colour $N=2$, associated to the Jones and Alexander polynomials} \label{52}
In this case, we do not remove any green points and work directly with submanifolds given by products of circles. We notice that in the above topological context associated to $N=2$, the product of circles around symmetric punctures gives a well defined homology class in the covering associated to the local system $\phi^{-n,0}$. This comes from the following remark.
\begin{remark}[Configurations on circles for $N=2$]

\

If we consider one circle around two punctures, it lifts towards a closed curve in the covering. 
This is due to the fact that we have no relative twisting on such a circle (since we have just one particle), and also the local system counts the monodromy of the loop given by this circle around the symmetric points $(k,2n-k)$ with opposite signs (as shown in the right hand side of figure \ref{Picturemon}). 
\end{remark}
Using this argument for all the circles, we conclude that $\bar{L}$ lifts to a well defined Lagrangian in the covering $\tilde{C}^{-n}_{2n,(n-1)(N-1)+1}$. 

So in this case we do not remove the maximal points of the circles and set $l=0$. We consider the configuration space $ C_{2n,(n-1)(N-1)+1}$ , with the local system $\phi^{-n}$ and the homology groups:
$H^{-n}_{2n,(n-1)(N-1)+1} \text{ and } H^{-n,\partial}_{2n,(n-1)(N-1)+1}.$
\begin{definition}[Homology classes for $N=2$]
Let us consider the homology classes which come from the geometric supports from figure \ref{JA}:
\begin{figure}[H]
\centering
$${\color{red} \mathscr F_{i_1,...,i_{n-1}} \in H^{-n}_{2n,(n-1)(N-1)+1}} \ \ \ \ \ \text{ and }\ \ \  \ \ \  {\color{dgreen} \mathscr L_{i_1,...,i_{n-1}}\in H^{-n,\partial}_{2n,(n-1)(N-1)+1}}.$$
\hspace{-5mm}\includegraphics[scale=0.37]{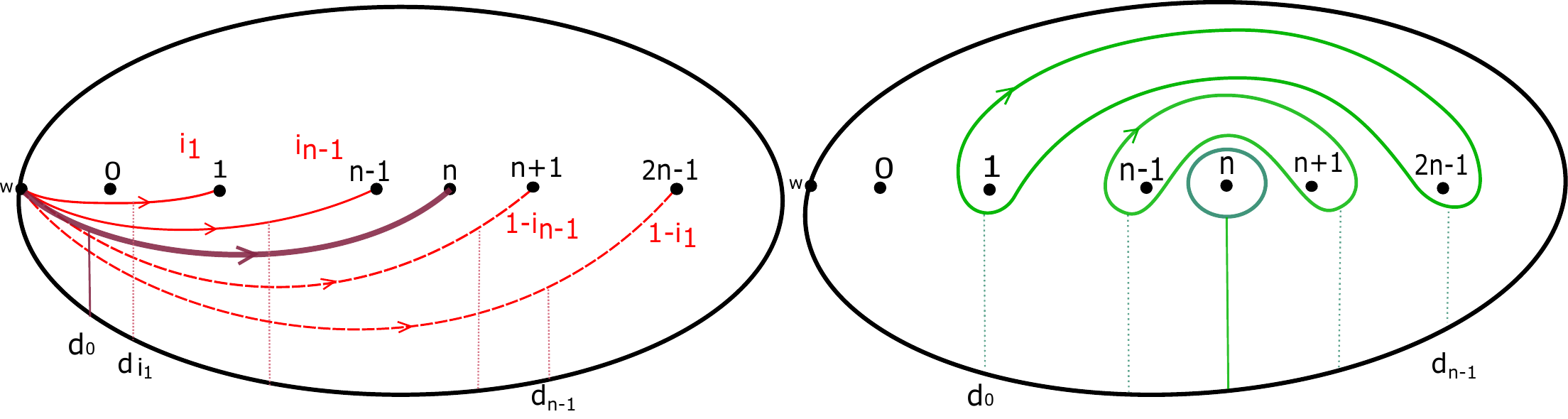}
\vspace{-13mm}
$$\hspace{15mm} {\color{purple!70!pink}\eta^{V}_{i_1,...,i_{n-1}}} \hspace{60mm} {\color{dgreen}\eta^L_{i_1,...,i_{n-1}}} \hspace{14mm}$$
%\vspace{-10mm}
\vspace{2mm}
\caption{Embedded classes, $N=2$, for Jones and Alexander polynomials}
\label{JA}
\end{figure}
\end{definition}
\begin{remark}
Even if the constructions of the homology classes for these embedded models for $N=2$ and $N>2$ are slightly different, the proof which we will provide for $N>2$ works in the same manner for the situation associated to $N=2$. This is due to the fact that even if we change the configuration space and the dual manifolds, in practice the intersection pairings will be calculated in the same manner.
\end{remark}
\section{Topological model with embedded Lagrangians}\label{6}
In this section we prove that the homology classes defined above are good building blocks for an intersection pairing that leads to the $U_q(sl(2))$-quantum invariants, through certain specialisations. More precisely, we aim to prove the intersection model presented in Theorem \ref{THEOREM}. We will use the following relation. 

\begin{lemma}[Relations between different Lagrangian intersections]\label{L:2}
For any $n$ and any braid $\beta_n\in B_n$, the following intersection pairings lead to the same result:
\begin{equation}\label{eq:50}
\mathscr S_N(\beta_n)=\mathscr I_N(\beta_n) \in \Z[x^{\pm1}, d^{\pm 1}].
\end{equation}
\end{lemma}
\begin{proof}
We start from the homology classes:
\begin{equation*}
\mathscr F_{i_1,...,i_{n-1}} \in H^{-n,n-1}_{2n,(n-1)(N-1)+1} \text { \ \ \ \ and \ \ \ \ \ }\mathscr L_{i_1,...,i_{n-1}} \in  H^{-n,n-1,\partial}_{2n,(n-1)(N-1)+1}.
\end{equation*}
Following the definition from equation \eqref{eq:00}, we have:
\begin{equation}
\mathscr S_N(\beta_n)=\sum_{i_1,...,i_{n-1}=0}^{N-1} \langle(\beta_{n} \cup {\mathbb I}_{n} ){ \color{black} \mathscr F_{i_1,...,i_{n-1}}}, {\color{black} \mathscr L_{i_1,...,i_{n-1}}}\rangle.
\end{equation}
On the other hand, the building blocks for the topological model from Theorem \ref{T:unified} were given by:
\begin{equation*}
\tilde{\mathscr U}_{0,i_1,...,i_{n-1},N-1-i_{n-1},...,N-1-i_{1}} \in H_{2n-1,(n-1)(N-1)} \text{ \ \ \ \ and \ \ \ \ \ } \mathscr G_n^N \in  H^{\partial}_{2n-1,(n-1)(N-1)}.
\end{equation*}
More precisely, following the definitions from equation \eqref{eq:1'} and \eqref{eq:2}  we have:
\begin{equation}
\begin{aligned}
&\mathscr I_N(\beta_n)=~\langle(\beta_{n} \cup {\mathbb I}_{n-1} ){ \color{black} \mathscr E_n^N}, {\color{black} \mathscr G_n^N}\rangle=\\
&=\sum_{{ i_1,...,i_{n-1}=0}}^{N-1} d^{-\sum_{k=1}^{n-1} i_k}\big\langle(\beta_{n} \cup {\mathbb I}_{n-1}) \tilde{ \mathscr U}_{0,i_1,...,i_{n-1},N-1-i_{n-1},...,N-1-i_{1}},\mathscr G_n^N\big\rangle.
\end{aligned}
\end{equation}
In the next part we will show that for any indices $i_1,...,i_{n-1}\in \{0,...,N-1\}$ the intersections between the associated classes which correspond to the immersed submanifolds and the classes corresponding to the embedded submanifolds are related as below:
\begin{equation}\label{eq:33}
\begin{aligned}
&\langle(\beta_{n} \cup {\mathbb I}_{n} ){ \color{black} \mathscr F_{i_1,...,i_{n-1}}}, {\color{black} \mathscr L_{i_1,...,i_{n-1}}}\rangle~=\\
&=d^{-\sum_{k=1}^{n-1} i_k}\langle(\beta_{n} \cup {\mathbb I}_{n-1}) \tilde{ \mathscr U}_{0,i_1,...,i_{n-1},N-1-i_{n-1},...,N-1-i_{1}},\mathscr G_n^N\rangle.
\end{aligned}
\end{equation}
This relation will be enough in order to conclude our Lemma. We will prove this in four main steps. 
\subsection{Step 1- Encoding the coefficient in the homology class}
\begin{lemma}[Twisting the base points \cite{Martel}]\label{L:1}
Using the properties of the local system, we can move the purple base point and pick up to a coefficient, as in the following formula:
\begin{figure}[H]
\centering
\hspace{-5mm}\includegraphics[scale=0.30]{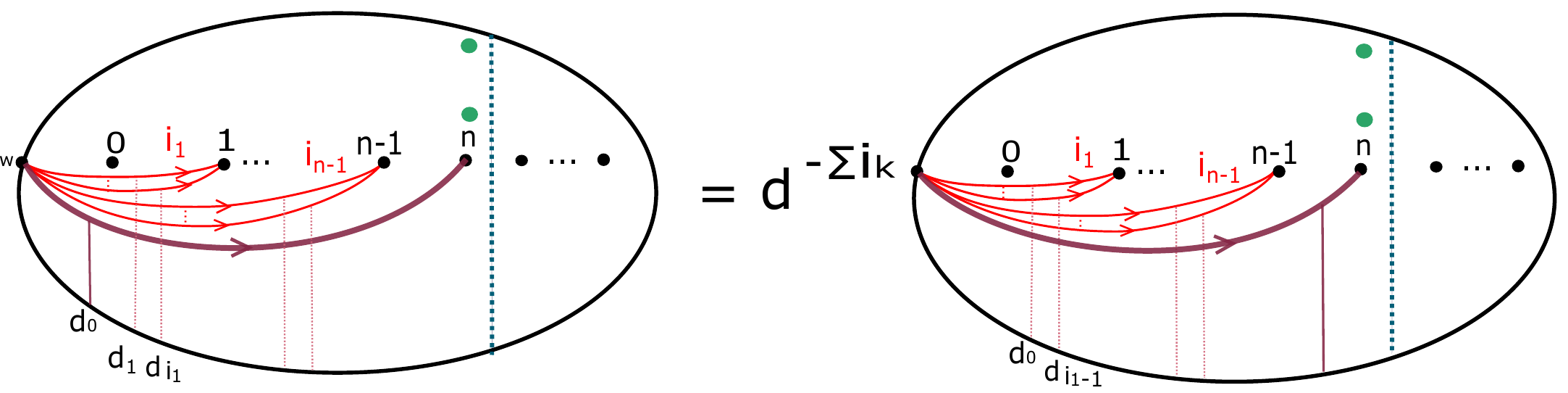}
\vspace{2mm}
\caption{Moving the purple base point}
\label{Picture1A}
\end{figure}
\end{lemma} 
\vspace{-6mm}
\begin{proof}
The two classes are constructed from the same geometric supports which are lifted using different paths to the base points. The coefficient from the formula comes from the evaluation of the local system on a certain loop, which corresponds to the ``difference'' between these two paths to the base points, from picture \ref{Picture1A}. This loop (in the configuration space) does not wind around the punctures but it twists the relative position of the points in the configuration $i_1+...+i_{n-1}$ times.
\end{proof}
Now we define an extra homology class, which comes from $\mathscr F_{i_1,...,i_{n-1}}$ but it is lifted in a different manner.
\begin{definition}[Change of the first homology class]

\

\noindent
Let $\tilde{\mathscr F}_{i_1,...,i_{n-1}} \in H^{-n,n-1}_{2n,(n-1)(N-1)+1} $ be the homology class given by the change of the lift, encoded by the path to the base points, presented in the picture below.
%%%%%%%%picture
\end{definition}
\begin{figure}[H]
\centering
$$ \mathscr F_{i_1,...,i_{n-1}} \in H^{-n,n-1}_{2n,(n-1)(N-1)+1} \ \ \ \ \ \ \  \rightarrow \ \ \ \ \ \ \ \tilde{\mathscr F}_{i_1,...,i_{n-1}}\in H^{-n,n-1}_{2n,(n-1)(N-1)+1}.$$
\hspace{-5mm}\includegraphics[scale=0.37]{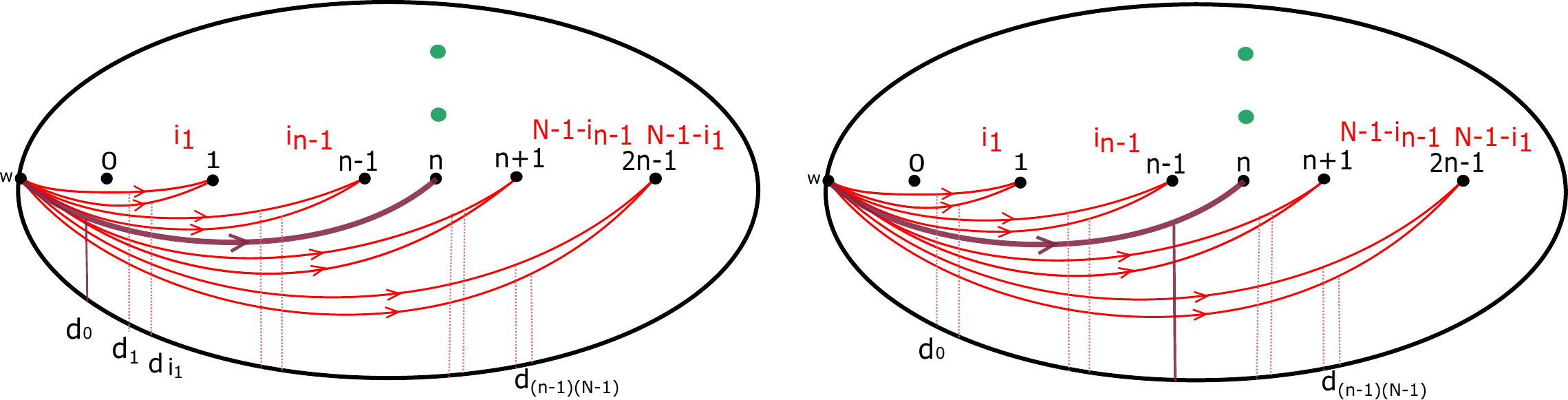}
\caption{Change of the base point}
\label{Picture1'}
\end{figure}
\vspace{-5mm}
Following Lemma \ref{L:1} we have that:
\begin{equation}
\mathscr F_{i_1,...,i_{n-1}}=d^{-\sum_{k=1}^{n-1} i_k}\tilde{\mathscr F} _{i_1,...,i_{n-1}}. 
\end{equation}
This property combined with relation \eqref{eq:33} show that it is enough to prove the following:
\begin{equation}\label{eq:3}
\begin{aligned}
&\langle(\beta_{n} \cup {\mathbb I}_{n} ){ \color{black} \tilde{\mathscr F}_{i_1,...,i_{n-1}}}, {\color{black} \mathscr L_{i_1,...,i_{n-1}}}\rangle=\\
&=\langle(\beta_{n} \cup {\mathbb I}_{n-1}) \tilde{ \mathscr U}_{0,i_1,...,i_{n-1},N-1-i_{n-1},...,N-1-i_{1}},\mathscr G_n^N\rangle.
\end{aligned}
\end{equation}
\subsection{Step 2- Change of the homology groups}
So far we have managed to encode the $d-$coefficient in the class $\tilde{\mathscr F}_{i_1,...,i_{n-1}}$. In this step we show that we can ``remove'' the $1$-manifolds which go around/ finish in the puncture labeled by $n$ from the geometric supports of the homology classes, without changing the intersection form. We start with the following definition.
\begin{definition}[Removing the $(n+1)^{st}$ puncture]\label{rempunct}

\

\noindent
Let us consider the homology classes which come from the geometric supports of $\tilde{\mathscr F}_{i_1,...,i_{n-1}}$ and ${\mathscr L}_{i_1,...,i_{n-1}}$ by removing the $1$-dimensional part which ends/ goes around the $(n+1)^{st}$ puncture (as in figure \ref{Picture1'}), and denote them as below:
\vspace{-4mm}
\begin{figure}[H]
\centering
 $$\mathscr F'_{i_1,...,i_{n-1}}\in H^{-(n-1),n-1}_{2n-1,(n-1)(N-1)} \ \ \ \ \ \text{ and } \ \ \ \ \ \mathscr L'_{i_1,...,i_{n-1}}\in H_{2n-1,(n-1)(N-1)}^{-(n-1),n-1,\partial}.$$
\hspace{-5mm}\includegraphics[scale=0.37]{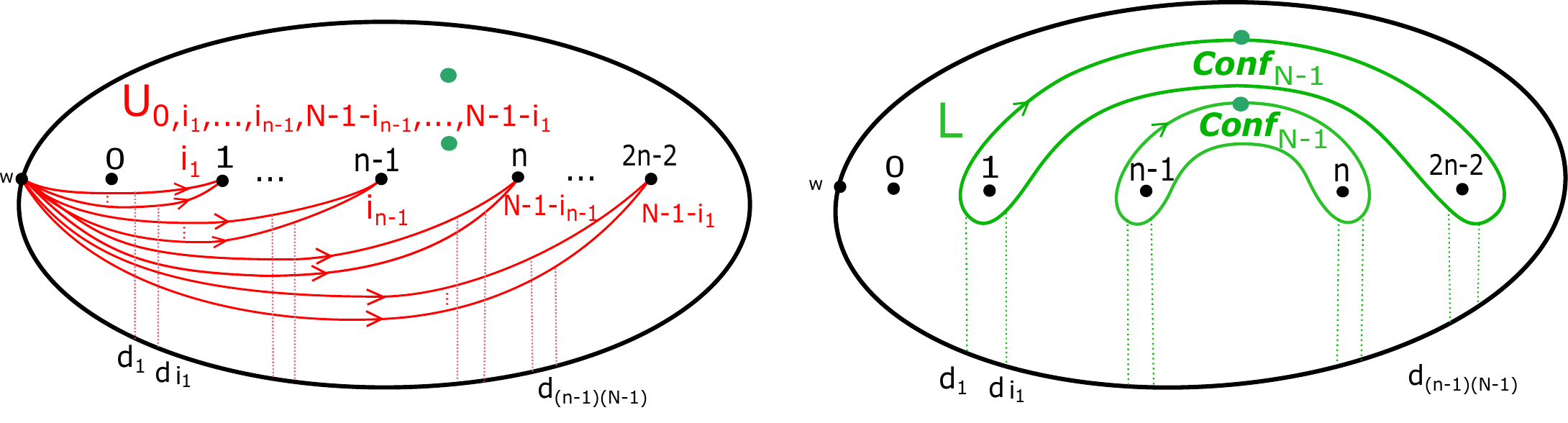}
\vspace{-3mm}
\caption{Submanifolds in $C_{3n-2,(n-1)(N-1)}$}
\label{Picture1}
\end{figure}
\end{definition}
\vspace{-8mm}
\begin{proposition} This change does not affect the value of the intersection pairing so we have:
\begin{equation}\label{eq:4}
\begin{aligned}
&\langle(\beta_{n} \cup {\mathbb I}_{n} ){ \color{black} \tilde{\mathscr F}_{i_1,...,i_{n-1}}}, {\color{black} \mathscr L_{i_1,...,i_{n-1}}}\rangle=\\
&=\langle(\beta_{n} \cup {\mathbb I}_{n-1}) { \color{black} \mathscr F'_{i_1,...,i_{n-1}}}, {\color{black} \mathscr L'_{i_1,...,i_{n-1}}}\rangle, \forall \beta_n \in B_n.
\end{aligned}
\end{equation}
\begin{proof}
We start with the intersection form from the left hand side and notice that the action of $(\beta_n \cup \mathbb I_n)$ does not modify either the red segment which goes between $w$ and the $(n+1)^{st}$ puncture or its corresponding path to the base point. Moreover, there exists an open neighbourhood of this puncture such that the geometric supports on the punctured disk which give the submanifolds $\bar{V}_{i_1,...,i_{n-1}}$ and $\bar{L}$ intersect in a unique point $y$ in this neighbourhood. Any intersection point $\bar{x}$ between $\bar{V}_{i_1,...,i_{n-1}}$ and $\bar{L}$ will contain $y$ as one of its components. Moreover, the loop associated to $\bar{x}$ in $C_{3n-1,(n-1)(N-1)+1}$ will have a component which is a loop in $C_{3n-2,(n-1)(N-1)}$ and another component given by a loop in the punctured disc, which is associated to the point $y$--denoted by $l_y$. This is constructed by the recipe from $\eqref{eq:1}$. Looking at the picture we notice that $l_y$ does not wind around any punctures and also it does not have relative winding in the configuration space with the other components of $l_{\bar{x}}$. This shows that $y$ does not contribute to the grading of the point $\bar{x}$.
This remark together with the method of computation of the pairing presented in equation \eqref{eq:1} conclude the relation between the intersection pairings.
\end{proof}
 
\end{proposition}
\subsection{Step 3- Twisting the local system}
In this part, we show that we can (un)twist the local system from $\phi^{-(n-1),n-1}$ back to $\phi^{0}$ if we pay the price of replacing the circles by figure eights. 
\begin{definition}[Change of the second homology class]
We consider the class ${\mathscr L}^{I'}_{i_1,...,i_{n-1}}\in H^{\partial}_{2n-1,(n-1)(N-1)}$ given by the lift of the product of the configuration spaces of $(N-1)$ points on the figure eights in $C_{2n-1,(n-1)(N-1)}$ (denoted by $G_n^N$), following the lift of the path from figure \ref{Picture5}.
\end{definition}
\begin{remark}
Since we changed back the local system, the configuration spaces on figure eights will lift to submanifolds in the covering $\tilde{C}_{2n-1,(n-1)(N-1)}$.

Moreover, we notice that the class ${\mathscr L}^{I'}_{i_1,...,i_{n-1}}\in H_{2n-1,(n-1)(N-1)}$ differs from $\mathscr G_{n}^N$ just by an element from the deck transformations, which comes from our choices of different lifts to the covering.
\end{remark}
\begin{definition}[Change of the first homology class] \label{D:change}

\

\noindent
Let ${\mathscr F}^{I'}_{i_1,...,i_{n-1}}\in H_{2n-1,(n-1)(N-1)}$ be the class given by the same geometric support in the base configuration space as $\mathscr F'_{i_1,...,i_{n-1}}$, lifted in the covering $\tilde{C}_{2n-1,(n-1)(N-1)}$, as in the picture below.
\vspace{-3mm}
\begin{figure}[H]
\centering
 $$ {\mathscr F}^{I'}_{i_1,...,i_{n-1}} \in H_{2n-1,(n-1)(N-1)} \ \ \ \text{ and }  \ \ \ {\mathscr L}^{I'}_{i_1,...,i_{n-1}}\in H^{\partial}_{2n-1,(n-1)(N-1)}.$$
\hspace{-3mm}\includegraphics[scale=0.35]{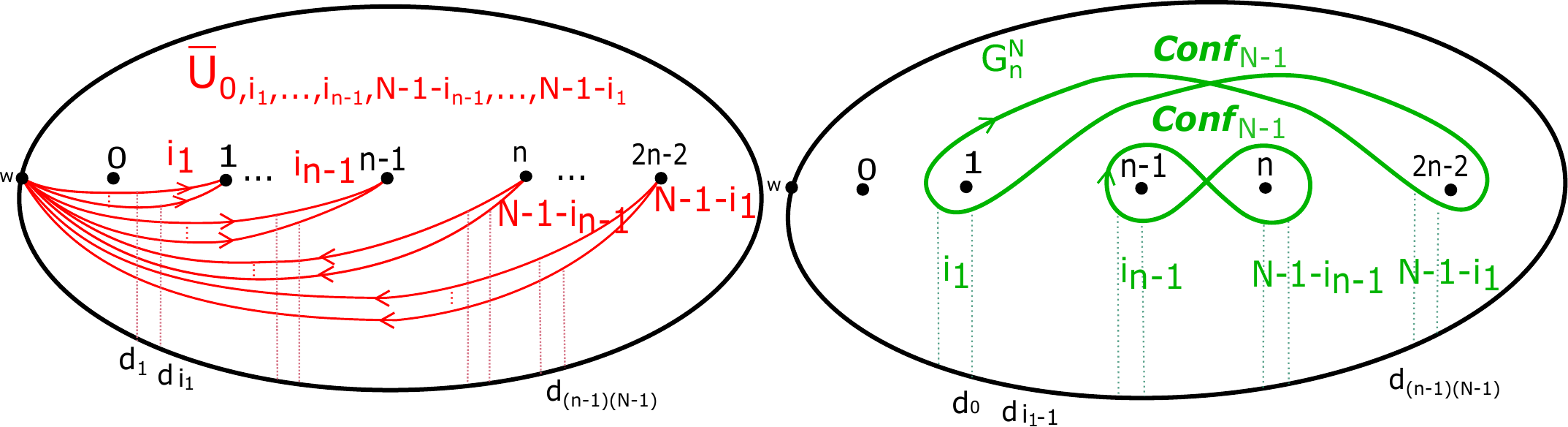}
\caption{Immersed submanifolds with simple paths to the base point}
\label{Picture5}
\end{figure}
\end{definition}
\begin{proposition}\label{P:4} 
We have the following relation, for any braid $\beta_n \in B_n$:
\begin{equation}\label{eq:5}
\begin{aligned}
&\langle(\beta_{n} \cup {\mathbb I}_{n-1} ){ \color{black} \mathscr F'_{i_1,...,i_{n-1}}}, {\color{black}  \mathscr L'_{i_1,...,i_{n-1}}}\rangle=\\
&=\langle(\beta_{n} \cup {\mathbb I}_{n-1}) { \color{black} {\mathscr F}^{I'}_{i_1,...,i_{n-1}}}, {\color{black} {\mathscr L}^{I'}_{i_1,...,i_{n-1}}}\rangle.
\end{aligned}
\end{equation}
\end{proposition}
\begin{proof}
We notice that the geometric supports of the classes: 
$$\mathscr F'_{i_1,...,i_{n-1}} \hspace{10mm} \text { and } \hspace{10mm} {\mathscr F}^{I'}_{i_1,...,i_{n-1}}$$ 
$$\mathscr L'_{i_1,...,i_{n-1}} \hspace{10mm} \text { and }  \hspace{10mm} {\mathscr L}^{I'}_{i_1,...,i_{n-1}}$$ 
are different just in the right half of the disc (containing the last $(n-1)$ punctures). Also, the local systems $\phi^{-(n-1),n-1}$ and $\phi^{0}$ evaluate the monodromies around the first punctures in the same way, but they differ because they count the monodromies around the last $(n-1)$ punctures with opposite orientations. 
On the other hand, the action of the braid group $B_n \cup {\mathbb I}_{n-1}$ is trivial on the half of the disc from the right hand side. 

Now, we look at the signs of the local intersections. We remark that the potential difference between the orientations of ${\bar U}_{0,i_1,...,i_{n-1},N-1-i_{n-1},...,N-1-i_{1}}$ and ${U}_{0,i_1,...,i_{n-1},N-1-i_{n-1},...,N-1-i_{1}}$ could occur because of the difference of the orientations of the red segments which end in the last $n-1$ punctures of the punctured disc. However, the same difference appears when we intersect these two manifolds with the dual supports $G_n^N$ and $L$ respectively. We conclude that overall the signs which come from these intersections will be positive in both cases. Concerning the intersections which appear in the left hand side on the disc, they will be literally the same in both situations, since the geometric supports coincide in this part of the picture.
%Concerning the sign of the local intersections, since we have changed the orientations of the red segments  which end in the last $n-1$ punctures, their intersections with the circles will have the same orientations as the previous ones between the 

Following the definition of the intersection pairing in terms of intersection points and evaluations of the local systems and the remarks from above, we conclude that the two intersection pairings give the same result.
\end{proof}
\subsection{Step 4- Pairings associated to different lifts of the same geometric support}

\

We remind the definition of the geometric supports that led to the model with embedded Lagrangians (from Theorem \ref{T:unified}), as well as the paths which prescribe their lifts:
\vspace{-4mm}
\begin{figure}[H]
\centering
 $${ \tilde{\mathscr U}_{0,i_1,...,i_{n-1},N-1-i_{n-1},...,N-1-i_{1}} \in H_{2n-1,(n-1)(N-1)}} \ \ \ \  \text{ and }\ \ \ \   {\mathscr G_{n}^N \in H^{\partial}_{2n-1,(n-1)(N-1)}}.$$
\hspace{-5mm}\includegraphics[scale=0.35]{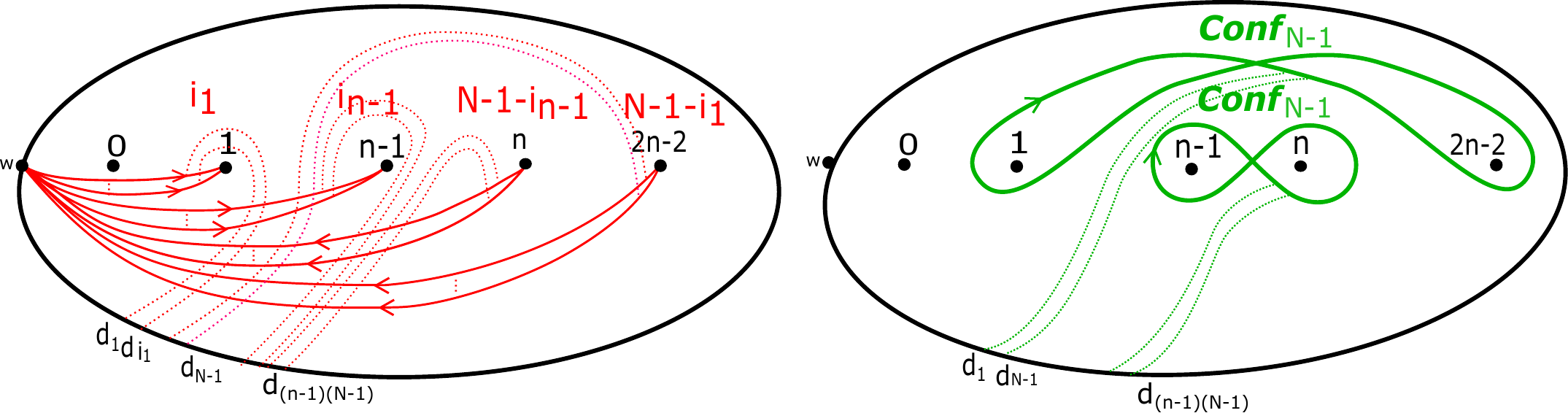}
\caption{Immersed classes with non-trivial paths to base points}
\label{Picture4}
\end{figure}
\vspace{-5mm}
Using the relation presented in equation \eqref{eq:3} together with the properties from equation \eqref{eq:4} and equation \eqref{eq:5}, we conclude that it is enough to show the following. 
\begin{lemma}\label{L:3}
There is an equality between the following intersection forms: 
\begin{equation}\label{eq:6}
\begin{aligned}
&\langle(\beta_{n} \cup {\mathbb I}_{n-1} ){ \color{black} {\mathscr F}^{I'}_{i_1,...,i_{n-1}}}, {\color{black}  {\mathscr L}^{I'}_{i_1,...,i_{n-1}}}\rangle=\\
&=\langle(\beta_{n} \cup {\mathbb I}_{n-1}) { \color{black} \tilde{ \mathscr U}_{0,i_1,...,i_{n-1},N-1-i_{n-1},...,N-1-i_1}}, {\color{black} { \mathscr G}^N_{n}} \rangle.
\end{aligned}
\end{equation}
%These are pairings between classes with the same geometric supports but with different paths to the base points.
\end{lemma}
\begin{proof}
The classes presented above are given by the same geometric submanifolds in the base configuration space, namely $$\bar{U}_{0,i_1,...,i_{n-1},N-1-i_{n-1},...,N-1-i_1} \text { and } G_n^N,$$ but they differ by the corresponding choices of the lifts. In the next part we will show that this difference does not affect the intersection pairing. 

We start with the classes $${\mathscr F}^{I'}_{i_1,...,i_{n-1}} \ \ \text{ and } \ \ \tilde{\mathscr U}_{0,i_1,...,i_{n-1},N-1-i_{n-1},...,N-1-i_1}.$$ Since they have the same geometric support in the configuration space, the two classes differ by an element of the deck transformations. More precisely, it exists a monomial in the ring of Laurent polynomials $\gamma(i_1,...,i_{n-1})\in \Z[x^{\pm 1}, d^{\pm 1}]$ such that:
\begin{equation}\label{eq:8}
{\mathscr F}^{I'}_{i_1,...,i_{n-1}}=\gamma(i_1,...,i_{n-1}) \cdot \tilde{ \mathscr U}_{0,i_1,...,i_{n-1},N-1-i_{n-1},...,N-1-i_1}.
\end{equation}
%(here, the sign comes from the change of orientations).
A similar argument shows that it exists a monomial $\gamma'(i_1,...,i_{n-1}) \in \Z[x^{\pm 1}, d^{\pm 1}]$ such that:
\begin{equation}\label{eq:8'}
{\mathscr L}^{I'}_{i_1,...,i_{n-1}}=\gamma'(i_1,...,i_{n-1}) \cdot \mathscr G_n^N.
\end{equation}

First, we investigate the pairings between the above classes, without any braid group action. We use the following result, presented in details in Section 7 of \cite{Cr3}. 
\begin{proposition}(\cite{Cr3})
The intersection pairing has the following property:
$$\langle\tilde{ \mathscr U}_{0,i_1,...,i_{n-1},N-1-i_{n-1},...,N-1-i_1}, { \mathscr G}^N_{n}\rangle=1.$$
\end{proposition}
In the next part, we show that the same property holds for the new classes. 
\begin{proposition} We have the following formula for the intersection pairing:
\begin{equation}\label{eq:9}  
\langle{\mathscr F}^{I'}_{i_1,...,i_{n-1}}, {\mathscr L}^{I'}_{i_1,...,i_{n-1}}\rangle=1.
\end{equation}
\end{proposition}

\

$$\hspace{33mm}\color{dgreen}J_1$$\\

\

$$\hspace{20mm}\color{dgreen} J_{n-1}$$
\vspace{-42mm}
\begin{figure}[H]
\centering
\includegraphics[scale=0.5]{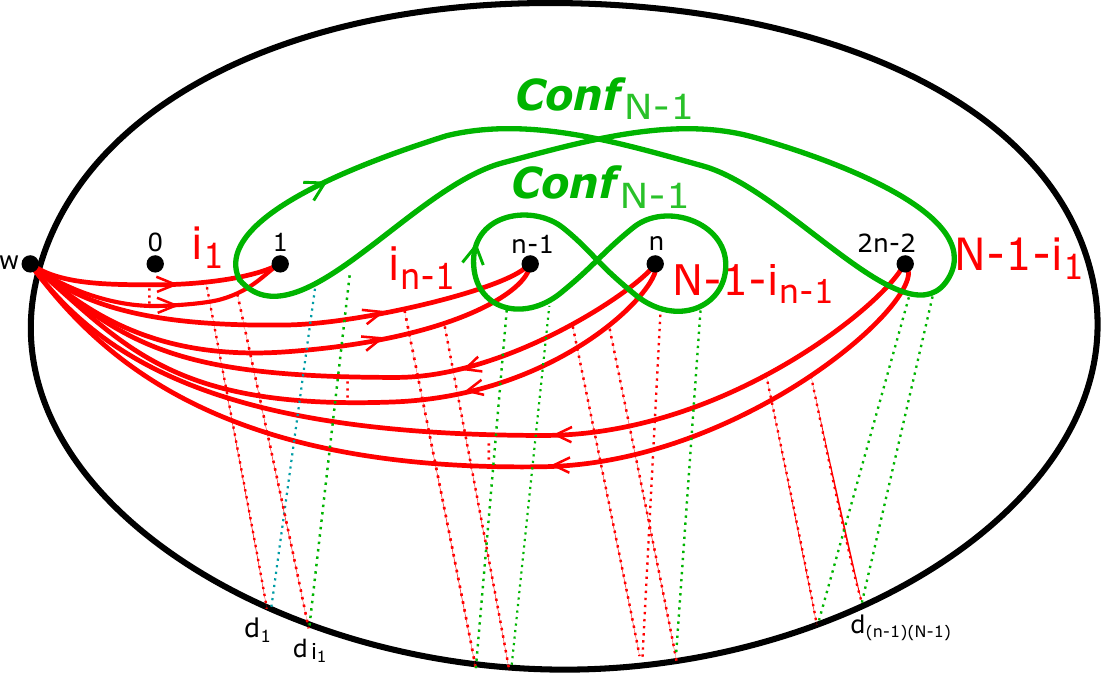}
\vspace{3mm}
\caption{Immersed submanifolds with simple paths to basepoints}
%\label{fig5}
\end{figure}
\vspace{-7mm}
\begin{proof} 
 Following formula \eqref{eq:1}, the intersection pairing is encoded by the intersection points between the geometric supports associated to ${\mathscr F}^{I'}_{i_1,...,i_{n-1}}$ and ${\mathscr L}^{I'}_{i_1,...,i_{n-1}}$. In order to have such an intersection point in the configuration space, we should choose on each figure eight  $(N-1)$ points at the intersection with the red segments. We notice that 
%we have in total $$i_0+...+i_{n-1}+(N-1-i_1)+...+(N-1-i_{n-1})=(n-1)(N-1)$$ such segments and
 there are no red segments or figure eights which end or go around the puncture labeled by $0$. This means that an intersection point between these geometric supports in the configuration space is given by the following description. For each $k\in \{1,..,n-1\}$, we should choose a natural number $j_k \in \N$ together with a collection of points as below:
\begin{enumerate}
\item[$\bullet$]$j_k$ points at the intersection between the red segments which end in the $k^{th}$ puncture and $J_k$.
\item[$\bullet$]$N-1-j_k$ points at the intersection between the red segments which end in the $(2n-1-k)^{th}$ puncture and $J_k$.
\end{enumerate}
Since we have exactly $i_k$ red segments which end in the $k^{th}$ puncture and $N-1-{i_k}$ red segments which end at the $(2n-1-k)^{th}$ puncture, it follows that:
$$i_k=j_k, \forall \ k \in \{1,...,n-1\}.$$
We conclude that the geometric supports intersect in a unique intersection point, which we denote by $$y\in \bar{ U}_{0,i_1,...,i_{n-1},N-1-i_{n-1},...,N-1-i_1} \cap G_n^N.$$ Now, looking at the local intersections we see that:
\begin{equation}
\begin{cases}
\alpha_y=1\\
\phi(l_{y})=1.
\end{cases}
\end{equation}
This shows that:
\begin{equation*}
\langle{\mathscr F}^{I'}_{i_1,...,i_{n-1}}, {\mathscr L}^{I'}_{i_1,...,i_{n-1}}\rangle=\alpha_y \cdot \phi(l_{y})=1.
\end{equation*}
and so ${\mathscr L}^{I'}_{i_1,...,i_{n-1}}$ has the right geometric property that we need and this concludes the Lemma. 
\end{proof}
Now, following \eqref{eq:8}, \eqref{eq:8'} and the intersection from \eqref{eq:9} we have:
\begin{equation}  
\langle \gamma(i_1,...,i_{n-1})  \cdot \tilde{ \mathscr U}_{0,i_1,...,i_{n-1},N-1-i_{n-1},...,N-1-i_1}, \gamma'(i_1,...,i_{n-1}) \cdot { \mathscr G}^N_{n}\rangle=1.
\end{equation}
Using the property that the form $\langle~ , ~\rangle$ is sesquilinear, we obtain the following:
\begin{equation}  
\begin{aligned}
\gamma(i_1,...,i_{n-1}) & \cdot  \left( \gamma'(i_1,...,i_{n-1})\right)^{-1} \cdot \\
& \langle \tilde{ \mathscr U}_{0,i_1,...,i_{n-1},N-1-i_{n-1},...,N-1-i_1}, { \mathscr G}^N_{n} \rangle=1.
\end{aligned}
\end{equation}
This equation combined with relation \eqref{eq:9} shows that:
\begin{equation}  
\gamma(i_1,...,i_{n-1})  \cdot  \left( \gamma'(i_1,...,i_{n-1})\right)^{-1}=1
\end{equation}
 and so we have
\begin{equation}  
\gamma(i_1,...,i_{n-1}) = \gamma'(i_1,...,i_{n-1}).
\end{equation}
 Going back to the evaluation of the intersection pairings after we act with the braid  on the first component, we conclude the following:
  \begin{equation}
\begin{aligned}
&\langle(\beta_{n} \cup {\mathbb I}_{n-1} ){ {\mathscr F}^{I'}_{i_1,...,i_{n-1}}}, {  {\mathscr L}^{I'}_{i_1,...,i_{n-1}}}\rangle=\\
&=\langle \gamma(i_1,...,i_{n-1}) \cdot (\beta_{n} \cup {\mathbb I}_{n-1}) { \tilde{ \mathscr U}_{0,i_1,...,i_{n-1},N-1-i_{n-1},...,N-1-i_1}}, \\
& \hspace{80mm} \gamma(i_1,...,i_{n-1}) \cdot  {{ \mathscr G}^N_{n}}\rangle=\\
&=\gamma(i_1,...,i_{n-1}) \cdot \left( \gamma(i_1,...,i_{n-1})\right)^{-1}\\
& \hspace{30mm} \langle(\beta_{n} \cup {\mathbb I}_{n-1}) { \tilde{ \mathscr U}_{0,i_1,...,i_{n-1},N-1-i_{n-1},...,N-1-i_1}},  {{ \mathscr G}^N_{n}}\rangle=\\
&=\langle(\beta_{n} \cup {\mathbb I}_{n-1}) { \tilde{ \mathscr U}_{0,i_1,...,i_{n-1},N-1-i_{n-1},...,N-1-i_1}}, {{ \mathscr G}^N_{n}}\rangle.
\end{aligned}
\end{equation}
\end{proof}
This shows that relation \eqref{eq:6} is true for any braid $\beta_n \in B_n$ and concludes that relation \eqref{eq:50} holds.
%This result combined with Corollary \ref{R:1} shows the equality of the intersection forms from relation \eqref{eq:6}. This in turn concludes the proof of the main Lemma.
\end{proof}

In the next part we use this property in order to prove the model from Theorem \ref{THEOREM}. We remind the formula of the specialisation of coefficients from definition \ref{D:1''}:
\begin{equation*}
\begin{aligned}
&\psi_{c,q,\lambda}: \Z[u^{\pm 1},x^{\pm1},d^{\pm1}]\rightarrow \Z[q^{\pm 1},q^{\pm \lambda}]\\
& \psi_{c,q,\lambda}(u)= q^{c \lambda}; \ \ \psi_{c,q,\lambda}(x)= \qs^{2 \lambda}; \ \ \psi_{c,q,\lambda}(d)=\qs^{-2}.
\end{aligned}
\end{equation*}
The topological model with immersed Lagrangians from Theorem \ref{T:unified} together with the relation between the Lagrangian intersections from Lemma \ref{L:2}, lead to the following relations:
\begin{equation}
\begin{aligned}
J_N(L,q)&=^{\text{Th} \ \ref{T:unified}} q^{-(N-1)w(\beta_n)} \cdot \ q^{-(N-1)(n-1)} \ \ \mathscr I_N(\beta_n)|_{\psi_{q,N-1}}=\\
&=^{\text{Lemma }  \ref{L:2}}q^{-(N-1)w(\beta_n)} \cdot \ q^{-(N-1)(n-1)} \ \mathscr S_N(\beta_n)|_{\psi_{q,N-1}}=\\
&= \left(  u^{-w(\beta_n)} \cdot \ u^{-(n-1)}  \mathscr S_N(\beta_n)\right)|_{\psi_{1,q,N-1}}=\\
&=\Lambda_N(\beta_n)|_{\psi_{1,q,N-1}}.
\end{aligned}
\end{equation}

\begin{equation}
\begin{aligned}
\Phi_N(L,\lambda)&=^{\text{Th} \ \ref{T:unified}} {\xi_N}^{-(1-N)\lambda w(\beta_n)} \cdot \ {\xi_N}^{-\lambda(1-N)(n-1)} \ \ \mathscr I_N(\beta_n)|_{\psi_{\xi_N,\lambda}}=\\
&=^{\text { Lemma }\ref{L:2}} {\xi_N}^{-(1-N)\lambda w(\beta_n)} \cdot \ {\xi_N}^{-\lambda(1-N)(n-1)} \ \mathscr S_N(\beta_n)|_{\psi_{\xi_N,\lambda}}=\\
&=  \left({u}^{-w(\beta_n)} \cdot \ {u}^{-(n-1)}  \mathscr S_N(\beta_n)\right)|_{\psi_{1-N,\xi_N,\lambda}}=\\
&=\Lambda_N(\beta_n)|_{\psi_{1-N,\xi_N,\lambda}}.
\end{aligned}
\end{equation}
These conclude the topological model for coloured Jones and Alexander polynomials presented in Theorem \ref{THEOREM}.
\section{Computable model with immersed Lagrangians}\label{7}
In this section we prove the state model from Theorem \ref{THEOREMIM}, which is a modification of the immersed model from Theorem \ref{T:unified} (from \cite{Cr3}), using the configuration space with an extra particle $C_{2n,(n-1)(N-1)+1}$ and simple paths to the base points. The change of the paths makes this model suitable for computations, as we will see in the next section. Also, we encode geometrically in this model the $d$-coefficients which appear in the first homology class $\mathscr E_n^N$ (definition \ref{D:2}).  For this construction we use the local system $\phi^{0}$.
\begin{definition}[Immersed classes suitable for computations]

\

Let us define the classes 
%$\mathscr F^{I}_{i_1,...,i_{n-1}} \in H^{0}_{2n,(n-1)(N-1)+1}$ and $\mathscr L^{I}_{i_1,...,i_{n-1}}\in H^{-n,\partial}_{0,(n-1)(N-1)+1}$ 
given by the geometric supports together with the paths to the base points presented below:
\end{definition}
\vspace{-6mm}
\begin{figure}[H]
\centering
$${\color{red} \mathscr F^{I}_{i_1,...,i_{n-1}} \in H^{0}_{2n,(n-1)(N-1)+1}} \ \ \ \ \ \ \  \text{ and }\ \ \ \ \ \ \ \  {\color{dgreen} \mathscr L^{I}_{i_1,...,i_{n-1}}\in H^{0,\partial}_{0,(n-1)(N-1)+1}}.$$
\hspace{-5mm}\includegraphics[scale=0.37]{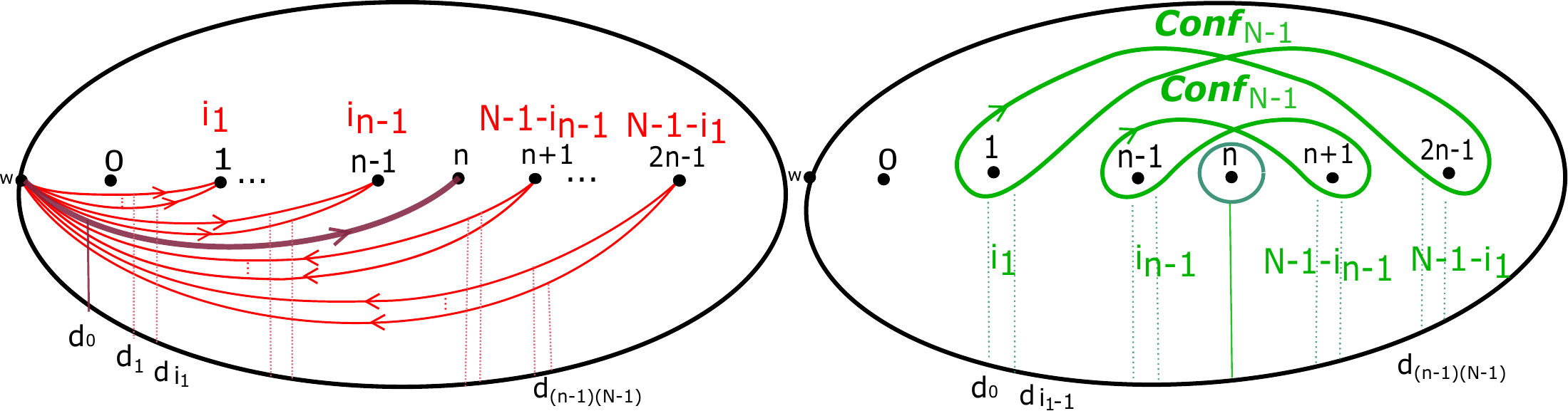}
\vspace{-1mm}
\caption{Classes immersed model-simple paths to the base points}
\label{Picture12}
\end{figure}
\vspace{-5mm}
Now we prove the unified model via states of immersed Lagrangian intersections presented in Theorem \ref{THEOREMIM}. We remind the state sum:
\begin{equation}\label{eq:0'}
\begin{aligned}
\Omega_N(\beta_n)&:=u^{-w(\beta_n)} u^{-(n-1)}\cdot\\
&\sum_{i_1,...,i_{n-1}=0}^{N-1} \langle(\beta_{n} \cup {\mathbb I}_{n} ){ \color{black} \mathscr F^{I}_{i_1,...,i_{n-1}}}, {\color{black} \mathscr L^{I}_{i_1,...,i_{n-1}}}\rangle \in \Z[u^{\pm},x^{\pm 1},d^{\pm 1}].
\end{aligned}
\end{equation}
Then, we will show that:
\begin{equation}
\begin{aligned}
&J_N(L,q)=\Omega_N(\beta_n)|_{\psi_{1,q,N-1}}\\
&\Phi_{N}(L,\lambda)=\Omega_N(\beta_n)|_{\psi_{1-N,\xi_N,\lambda}}.
\end{aligned}
\end{equation}
\begin{proof}
The proof of this result is based on the steps used for the proof of the embedded model from the previous section. Similarly to that situation, it is enough to show that:
\begin{equation}
\begin{aligned}
&\langle(\beta_{n} \cup {\mathbb I}_{n} ){ \mathscr F^{I}_{i_1,...,i_{n-1}}}, {\mathscr L^{I}_{i_1,...,i_{n-1}}}\rangle~=\\
&=d^{-\sum_{k=1}^{n-1} i_k}\langle(\beta_{n} \cup {\mathbb I}_{n-1}) \tilde{ \mathscr U}_{0,i_1,...,i_{n-1},N-1-i_{n-1},...,N-1-i_{1}},\mathscr G_n^N\rangle.
\end{aligned}
\end{equation}
Using Lemma \ref{L:1} in this context, associated to the local system $\phi^0$, we have that:
\begin{equation}
\mathscr F^{I}_{i_1,...,i_{n-1}}=d^{-\sum_{k=1}^{n-1} i_k} \cdot \tilde{\mathscr F}^{I}_{i_1,...,i_{n-1}}.
\end{equation}
Here $\tilde{\mathscr F}^{I}_{i_1,...,i_{n-1}}$ is the class obtained from the same geometric support and path to the base point as the ones used for $\tilde{\mathscr F}_{i_1,...,i_{n-1}}$ (from figure \ref{Picture1'}), but where we lift the submanifold to the covering associated to the local system $\phi^0$. Using this remark, we want to prove the following relation:
\begin{equation}\label{e:1}
\begin{aligned}
&\langle(\beta_{n} \cup {\mathbb I}_{n} ){ \tilde{\mathscr F}^{I}_{i_1,...,i_{n-1}}}, {\mathscr L^{I}_{i_1,...,i_{n-1}}}\rangle~=\\
&=~\langle(\beta_{n} \cup {\mathbb I}_{n-1}) \tilde{ \mathscr U}_{0,i_1,...,i_{n-1},N-1-i_{n-1},...,N-1-i_{1}},\mathscr G_n^N\rangle.
\end{aligned}
\end{equation}
Now, we look at the classes ${\mathscr F}^{I'}_{i_1,...,i_{n-1}}$ and ${\mathscr L}^{I'}_{i_1,...,i_{n-1}}$  which were introduced in definition \ref{D:change}. They are obtained from the geometric supports and paths to the base points of $\tilde{\mathscr F}^{I}_{i_1,...,i_{n-1}}$ and $\mathscr L^{I}_{i_1,...,i_{n-1}}$ by removing the curve/ circle which end or go around the $(n+1)^{st}$ puncture. We remark that this change does not modify the intersection form:
\begin{equation}\label{e:2}
\begin{aligned}
&\langle(\beta_{n} \cup {\mathbb I}_{n} ){ \tilde{\mathscr F}^{I}_{i_1,...,i_{n-1}}}, {\mathscr L^{I}_{i_1,...,i_{n-1}}}\rangle~=\\
&\langle(\beta_{n} \cup {\mathbb I}_{n-1} ){ {\mathscr F}^{I'}_{i_1,...,i_{n-1}}}, {{\mathscr L}^{I'}_{i_1,...,i_{n-1}}}\rangle.
\end{aligned}
\end{equation}
Following Lemma \ref{L:3}, we have:
\begin{equation}\label{e:3}
\begin{aligned}
&\langle(\beta_{n} \cup {\mathbb I}_{n-1} ){\mathscr F}^{I'}_{i_1,...,i_{n-1}}, {\mathscr L}^{I'}_{i_1,...,i_{n-1}}\rangle~=\\
&=~\langle(\beta_{n} \cup {\mathbb I}_{n-1}) \tilde{ \mathscr U}_{0,i_1,...,i_{n-1},N-1-i_{n-1},...,N-1-i_{1}},\mathscr G_n^N\rangle.
\end{aligned}
\end{equation}
Relations \eqref{e:2} and \eqref{e:3} show that equation \eqref{e:1} holds and this concludes the proof of this model.
\end{proof}
\section{Useful form for computations} \label{8}
In this section we present two formulas which we obtained along the way, in the proofs from the previous two sections, where we used classes given by geometric supports where we removed the circle and the extra point from the middle of the picture (which had the role of encoding the $d$-coefficients). This might be useful as well for the computational point of view.

 \subsection{Embedded case} We start with the embedded state sum model. 
We remind that in definition \ref{rempunct} for any $i_1,...,i_{n-1}\in \{0,...,N-1\}$ we constructed two classes denoted by: 
\vspace{-5mm}
\begin{center}
\begin{figure}[H]
\centering
$${\color{red} \mathscr F'_{i_1,...,i_{n-1}} \in H^{-(n-1),n-1}_{2n-1,(n-1)(N-1)}} \ \ \text{ and }\ \  {\color{dgreen} \mathscr L'_{i_1,...,i_{n-1}}\in H^{-(n-1),(n-1),\partial}_{2n-1,(n-1)(N-1)}}.$$
\includegraphics[scale=0.35]{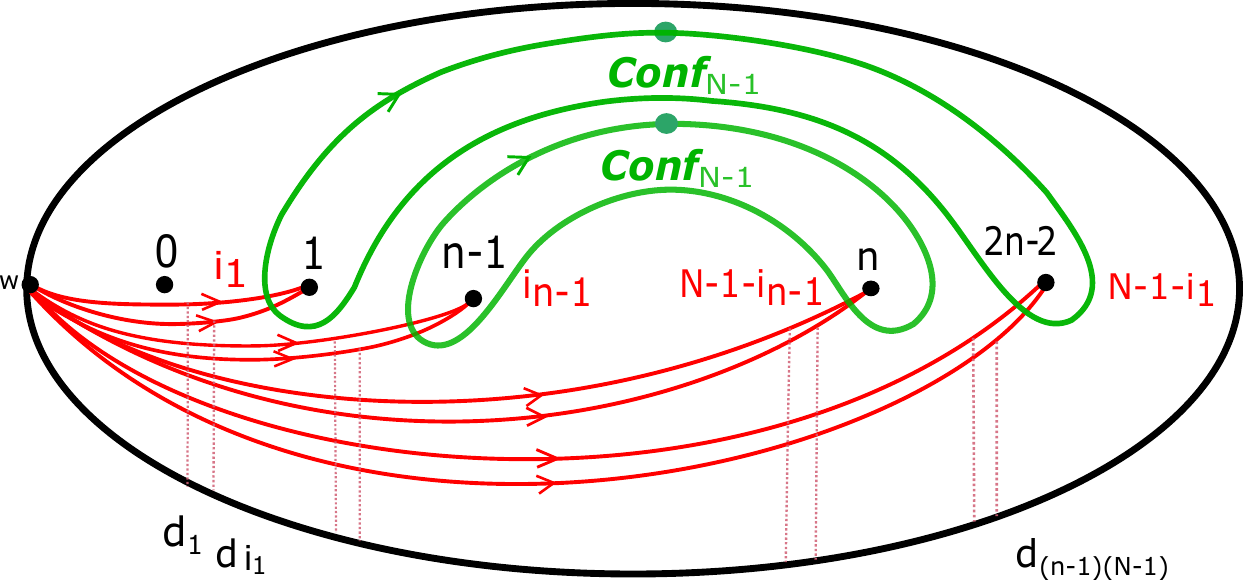}
%\caption{Embedded Lagrangians}
\end{figure}
\end{center}
\vspace{-7mm}
\begin{corollary}[Modified embedded state sum model]\label{CC:Computation''}
Let $L$ be an oriented link and $\beta_n \in B_n$ such that $L=\hat{\beta}_n$.  We define the polynomial:
\begin{equation}\label{C:Computation''}
\begin{aligned}
\Lambda'_N(\beta_n)(u,x,d)&:=u^{-w(\beta_n)} u^{-(n-1)} \sum_{i_1,...,i_{n-1}=0}^{N-1} d^{-\sum_{k=1}^{n-1}i_k} \\
&\langle(\beta_{n} \cup {\mathbb I}_{n-1} ){ \mathscr F'_{i_1,...,i_{n-1}}}, { \mathscr L'_{i_1,...,i_{n-1}}}\rangle\in \Z[u^{\pm1},x^{\pm 1},d^{\pm 1}].
\end{aligned}
\end{equation}
Then $\Lambda'_N$ gives the $N^{th}$ coloured Jones and $N^{th}$ coloured Alexander polynomials:
\begin{equation}
\begin{aligned}
&J_N(L,q)=\Lambda_N'(\beta_n)|_{\psi_{1,q,N-1}}\\
&\Phi_{N}(L,\lambda)=\Lambda'_N(\beta_n)|_{\psi_{1-N,\xi_N,\lambda}}.
\end{aligned}
\end{equation}
\end{corollary}
\subsection{Immersed case} Now, we present an immersed state sum model which follows from the proof of the main theorem. We remind that in definition \ref{D:change} we introduced the homology classes:
\begin{center}
\begin{figure}[H]
\centering
$${\color{red} {\mathscr F}^{I'}_{i_1,...,i_{n-1}} \in H^{-(n-1)}_{2n-1,(n-1)(N-1)}} \ \ \text{ and }\ \  {\color{dgreen} {\mathscr L}^{I'}_{i_1,...,i_{n-1}}\in H^{-(n-1),\partial}_{2n-1,(n-1)(N-1)}}.$$
\includegraphics[scale=0.35]{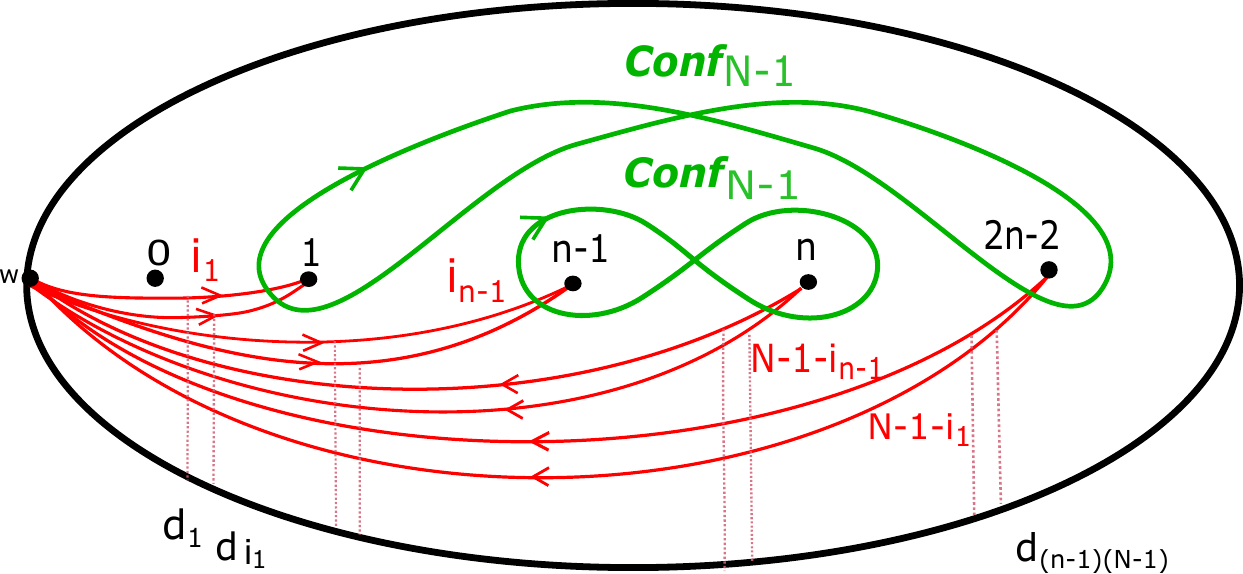}
%\caption{Embedded Lagrangians}
\end{figure}
\end{center}
\vspace{-7mm}
\begin{corollary}[Modified immersed state sum model]\label{CC:Computationimmersed}
Let $L$ be an oriented link such that $L=\hat{\beta}_n$.  We define the polynomial:
\begin{equation}\label{C:Computationimmersed}
\begin{aligned}
\Omega'_N(\beta_n)(u,x,d)&:=u^{-w(\beta_n)} u^{-(n-1)} \sum_{i_1,...,i_{n-1}=0}^{N-1} d^{-\sum_{k=1}^{n-1}i_k} \\
&\langle(\beta_{n} \cup {\mathbb I}_{n-1} ){ {\mathscr F}^{I'}_{i_1,...,i_{n-1}}}, { {\mathscr L}^{I'}_{i_1,...,i_{n-1}}}\rangle\in \Z[u^{\pm1},x^{\pm 1},d^{\pm 1}].
\end{aligned}
\end{equation}
Then $\Omega'_N$ gives the $N^{th}$ coloured Jones and $N^{th}$ coloured Alexander polynomials:
\begin{equation}
\begin{aligned}
&J_N(L,q)=\Omega_N'(\beta_n)|_{\psi_{1,q,N-1}}\\
&\Phi_{N}(L,\lambda)=\Omega'_N(\beta_n)|_{\psi_{1-N,\xi_N,\lambda}}.
\end{aligned}
\end{equation}
\end{corollary}
\subsection{Example: Unified model for the Jones and Alexander polynomials} \label{E:Example}
Following the intersection models from above we see that for the particular case where $N=2$ the polynomial $\Lambda_2$ specialises to both Alexander polynomial and Jones polynomials (Theorem \ref{THEOREM}), and this is true also for $\Omega_2$ (Theorem \ref{THEOREMIM}). 

Now, we see an example of computations. We will use the formulas from Corollary \ref{CC:Computation''} and Corollary \ref{CC:Computationimmersed}.

 Let $T$ be the trefoil knot, which is the closure of $\sigma^3$, where $\sigma \in B_2$.
In this case, we have $n=2$. In the state sum from formula \eqref{C:Computation''} we have two intersections, which correspond to the index $i_1=1$ (Case I) or $i_1=0$ (Case II). We draw the associated geometric supports in figure \ref{Tembedded}. The weights of the intersection points come from the local system evaluated on the loops which are constructed following the recipe from formula \eqref{eq:1}. For the intersection points which are on the left hand side of the disc, we have an extra weight which multiplies by $d^{-1}$ (this comes from the $d$-coefficients from equation \eqref{C:Computation''}). 
 
We obtain: 
\begin{equation}
\begin{aligned}
\Lambda'_2(\sigma^3)(u,x,d)&= u^{-4} \left( d^{-1}\langle( \sigma^3 \cup {\mathbb I}_{1} ) \mathscr F'_{1}, \mathscr L'_{1}\rangle+ \langle( \sigma^3 \cup {\mathbb I}_{1} )  \mathscr F'_{0},  \mathscr L'_{0}\rangle \right)=\\
&=u^{-4}\left( d^{-1}(-x^{-3}+x^{-2}-x^{-1}+1)+1\right).
\end{aligned}
\end{equation}
For the Jones polynomial, we have to specialise: $u=q, \ x=q^{2},\ d=q^{-2}$ and we get the normalised version of this invariant:
\begin{equation}
\begin{aligned}
J_2(T,q)=\Lambda'_2(\sigma^3)(u,x,d)|_{\psi_{1,q,1}}= -q^{-8}+q^{-2}+q^{-6}.
\end{aligned}
\end{equation}
Dually, if we want to get the Alexander polynomial, we have to specialise using the variables: $u={\xi_2}^{-\lambda}, \ x={\xi_2}^{2\lambda}, \ d={\xi_2}^{-2}=-1$ and we obtain:
\begin{equation}
\begin{aligned}
\Phi_2(T,\lambda)=\Lambda'_2(\sigma^3)(u,x,d)|_{\psi_{-1,\xi_2=i,\lambda}}={\xi_2}^{2\lambda}-1+{\xi_2}^{-2\lambda}.
\end{aligned}
\end{equation}
Replacing the variable ${\xi}^{2\lambda}_2$ with $x$, we get the usual form of the Alexander polynomial:
\begin{equation}
\begin{aligned}
\Delta(T,x)=x-1+x^{-1}.
\end{aligned}
\end{equation}
\begin{remark}
In our model we work with the normalised version of coloured Jones polynomials,  which means that for the particular case when $N=2$ we recover $J_2(T,q)$ which is the reduced Jones polynomial (see \cite{BN} Section 2).  
\end{remark}

Now we compute this example using the immersed model. In this case, we have also two intersection pairings and we drew the associated intersections between the geometric supports in picture \ref{Timmersed}.
Following the immersed model formula from equation \eqref{C:Computationimmersed} we obtain:
\begin{equation}
\begin{aligned}
\Omega'_2(\sigma^3)(u,x,d)&= u^{-4} \left(d^{-1} \langle( \sigma^3 \cup {\mathbb I}_{1} ) {\mathscr F}^{I'}_{1}, \mathscr L^{I'}_{1}\rangle+ \langle( \sigma^3 \cup {\mathbb I}_{1} )  \mathscr F^{I'}_{0},  \mathscr L^{I'}_{0}\rangle \right)=\\
&=u^{-4}\left( d^{-1}(-x^{-3}+x^{-2}-x^{-1}+1)+1\right).
\end{aligned}
\end{equation}
Specialising in the appropiate way, we conclude that this model also recovers the Jones and Alexander polynomials of the trefoil knot:
\begin{equation}
\begin{aligned}
&J_2(T,q)=\Omega'_2(\sigma^3)(u,x,d)|_{\psi_{1,q,1}}= -q^{-8}+q^{-2}+q^{-6}.\\
&\Delta(T,x)=\Omega'_2(\sigma^3)(u,x,d)|_{\left(\psi_{-1,\xi_2,\lambda}; x={\xi_2}^{2\lambda}\right)}=x-1+x^{-1}.
\end{aligned}
\end{equation}
\clearpage

\begin{center}
$$ \text{Immersed state sum} \hspace{35mm} \text{ Embedded state sum}$$
$$ \ \ \Lambda'_2(\sigma^3) \hspace{55mm} \Omega'_2(\sigma^3)$$
$J_2(T,q)   \ \  \ \ \ \ \ \ \ \ \Delta(T,x) \hspace{33mm} J_2(T,q)   \ \ \ \ \ \ \  \ \ \ \Delta(T,x)$

\

$ \ \ \text{ Jones}  \ \ \ \ \  \ \ \ \ \  \ \text {Alexander} \hspace{30mm} \ \ \text{ Jones}  \ \ \ \ \ \ \ \ \ \ \ \text {Alexander}$\\
$\text{ polynomial }   \ \ \ \ \ \ \text {polynomial} \hspace{28mm} \text{ polynomial } \ \ \ \  \ \text {polynomial}$
\begin{figure}[H]
    \centering
    \begin{minipage}{0.45\textwidth}
        \centering
        \includegraphics[width=1\textwidth]{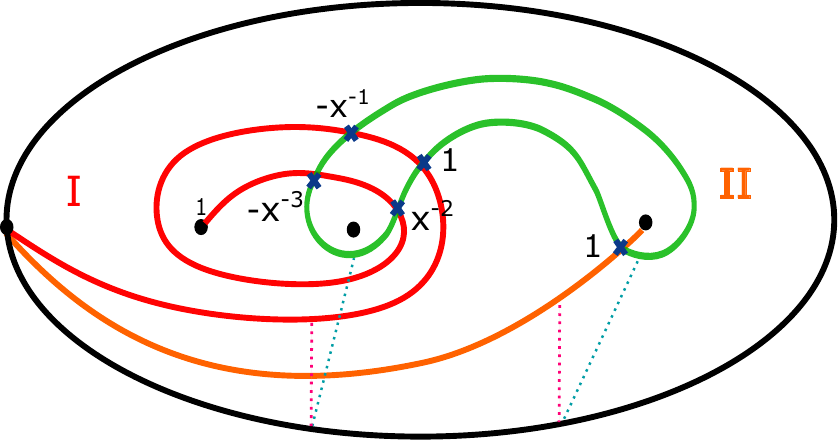} % first figure itself
        \caption{Trefoil knot: Embedded picture}
        \label{Tembedded}
    \end{minipage}\hfill
    \begin{minipage}{0.45\textwidth}
        \centering
        \includegraphics[width=1\textwidth]{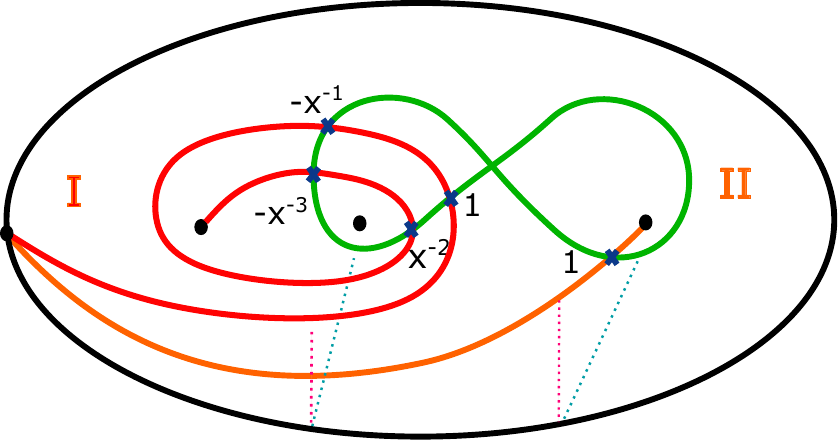} % second figure itself
        \caption{Trefoil knot: Immersed picture}
        \label{Timmersed}  
    \end{minipage}
\end{figure}
\end{center}

\

\

\href{http://www.cristinaanghel.ro/}{www.cristinaanghel.ro}  

\end{document}